%% file: Final.tex
\documentclass[11pt,letterpaper]{amsart}

\usepackage[utf8]{inputenc}
\usepackage[T1]{fontenc}
\usepackage{amsmath, amssymb, amsthm, mathtools}
\usepackage{bbm} 
\usepackage{xcolor}

\usepackage[colorlinks=true, linkcolor=blue, citecolor=blue, urlcolor=blue]{hyperref}

\usepackage[margin=1.2in]{geometry}
\usepackage{tikz}
\usetikzlibrary{arrows.meta}

\DeclareMathOperator{\dist}{dist}
\DeclareMathOperator{\supp}{supp}

\DeclareMathOperator*{\essinf}{ess\,inf}
\providecommand{\abs}[1]{\left \lvert#1 \right \rvert} 
\providecommand{\norm}[1]{\left \lVert#1 \right \rVert}

\newcommand{\CM}{{\mathbb C}}
\newcommand{\NM}{{\mathbb N}}
\newcommand{\RM}{{\mathbb R}}

\newcommand{\TM}{{\mathbb T}}
\newcommand{\ZM}{{\mathbb Z}}

\theoremstyle{plain}
\newtheorem{theorem}{Theorem}[section]
\newtheorem{proposition}[theorem]{Proposition}
\newtheorem{lemma}[theorem]{Lemma}
\newtheorem{corollary}[theorem]{Corollary}

\newtheorem{assumption}{Assumption}

\theoremstyle{definition}
\newtheorem{definition}[theorem]{Definition}

\theoremstyle{remark}
\newtheorem{remark}[theorem]{Remark}

\theoremstyle{definition}
\newtheorem{example}[theorem]{Example}

\numberwithin{equation}{section}

\begin{document}
\sloppy
\title{(Sharp arithmetic) localization beyond transversality}

\author{Svetlana Jitomirskaya}
\address{Department of Mathematics, University of California,
Berkeley, California 94720}
\email{sjitomi@berkeley.edu}

\author{Netanel Levi}
\address{Department of Mathematics, University of California,
Irvine, CA 92697}
\email{netanell@uci.edu}

\author{Tom Stoiber}
\address{Department of Mathematics, University of California,
Berkeley, California 94720}
\address{Simons Institute for the Theory of Computing, Berkeley, CA 94720}
\email{tstoiber@berkeley.edu}

\begin{abstract}
We prove sharp arithmetic localization for one-dimensional quasiperiodic potentials with jump discontinuities, beyond uniform transversality.  Our results include every non-constant $C^1$ monotone sampling function, even with arbitrarily large flat regions, and also certain  non-monotone potentials. A key new mechanism is a symbolic unique-continuation principle: first-return words to the region where $f'>0$ have no common Dirichlet resonance, forcing eigenfunction mass to transfer away from flat or negatively sloped regions and recovering quantitative monotonicity of finite-volume eigenvalue branches.
\end{abstract}

\maketitle

\section{Introduction}
In this work we study localization for quasiperiodic Schr\"odinger
operators in the absence of quantitative monotonicity or transversality
assumptions on the sampling function. We consider operators acting on
$\ell^2(\mathbb Z)$ of the form
\[
 (H(x)\psi)(n)=\psi(n-1)+\psi(n+1)+f(x+n\alpha)\psi(n),
\]
where $f$ is $1$-periodic, and $x\in\mathbb R$ and
$\alpha\in\mathbb R\setminus\mathbb Q$ are the phase and frequency,
respectively.

A basic difficulty in proving localization is to control the set of
parameters for which a finite-volume operator has an eigenvalue close
to a prescribed energy. Quantitative monotonicity is one of the main
ways to obtain this control. In the standard Anderson model, increasing
a single site potential gives a positive rank-one perturbation. With a regular single-site law, this underlies the
usual Wegner bounds and spectral-averaging arguments. Its absence is
a known difficulty for
nonmonotone random operators, requiring special efforts to overcome, see Elgart--Shamis--Sodin~\cite{ess}.

For quasiperiodic operators, all potential values depend on the same
phase and cannot be averaged independently. Quantitative monotonicity
or transversality of $f$ provides a way to control resonances through the
dependence on this phase. One way to ensure transversality is Łojasiewicz type inequality for analytic $f,$  a property significantly exploited in various arguments since the work of Bourgain and collaborators (see \cite{bbook} and references therein.)   For smooth potentials beyond the analytic category, quantitative transversality has also played a role in all existing general proofs of quasiperiodic localization known to us, either in the localization argument itself or in establishing positivity of the Lyapunov exponent used as an input. This includes Gevrey \cite{eliasson, K}, and finitely smooth or cosine-type potentials \cite{CSZ,CSZ1,CSZ2,FSW,FV,GYZa2, sin}.

Transversality assumptions exclude in particular flat intervals. If $f$ is constant on an
interval, the
set of phases for which $f(x)$ lies within $\varepsilon$ of that
value has measure bounded away from zero as $\varepsilon\to0$.
Thus quantitative spectral variation cannot be read off from the
sampling function itself.

Here we consider monotone sampling functions, namely $1$-periodic
functions that are nondecreasing on $[0,1)$. For this class,
localization at all coupling strengths and Diophantine frequencies
was proved in~\cite{JK1}, with an extension to unbounded potentials
in~\cite{Kach}, under the assumption that $f$ is {\it anti-Lipschitz}:
\[
 f(y)-f(x)\geq\gamma(y-x),\qquad 0\leq x<y<1,
\]
for some $\gamma>0$. The same assumption is used in the sharp
arithmetic localization theorem of~\cite{JK2}. This strong form of transversality forces $f$ to grow at a definite linear rate everywhere, and in particular excludes all monotone potentials with flat pieces. This condition has been crucial for all proofs.

Here we prove localization without this quantitative monotonicity
assumption for every bounded, nonconstant $C^1$ monotone sampling
function, throughout the sharp arithmetic localization regime.
In particular, the sampling function may even be constant on most of the
period. The result holds for every irrational frequency.

To state the result, let $L(E)$ denote the Lyapunov exponent and let
\begin{equation}\label{eq:beta-def}
     \beta(\alpha):=\limsup_{n\to\infty}\frac{1}{q_n}\log q_{n+1},
 \end{equation}
 where $p_n/q_n$ are the continued fraction convergents of $\alpha$.
 
 For quasiperiodic operators exponential almost repetitions at $q_n$ corresponding to $\alpha$ with $\beta(\alpha)>0,$ are a delocalizing force that is in competition with the localizing force coming from positive  $L(E)$. 

 We say that operator exhibits \emph{sharp arithmetic transition in frequency} if  the set
     $
     \{E\in\mathbb{R}:L(E)>\beta(\alpha)\}
    $ 
     can only support pure point spectrum with exponentially decaying eigenfunctions, and the set $\{E:0<L(E)<\beta(\alpha)\}$ supports singular continuous spectrum. 
 
\begin{theorem}\label{thm:main1}
 Let $f:[0,1)\to\mathbb{R}$ be a bounded $C^1$ monotone function which is non-constant, and let $\alpha\in\mathbb{R}\setminus\mathbb{Q}$. Then the corresponding quasiperiodic family $H(x)$ exhibits a sharp arithmetic spectral transition in frequency for a.e. $x$. 
\end{theorem}
\begin{remark}\begin{enumerate}
\item As in \cite{JK2}, the localization part of the conclusion holds for all $x.$
\item
    The $C^1$ assumption in Theorem \ref{thm:main1} is imposed here only for simplicity. The actual hypotheses used in the proof are weaker, and allow points of discontinuity. See Section \ref{sec:ev_branches} and Theorem \ref{thm:monotonederivative_implies_A1A2} for more details; also see Theorem \ref{thm:main2} below.

\item Sharp arithmetic transition in frequency was originally conjectured for the almost Mathieu operator (AMO), $f(x)=2\lambda\cos(2\pi x)$ in \cite{Jit2}.  This conjecture was subsequently established in \cite{AYZ,JL2} (see also \cite{GYZa}), after a sequence of important earlier works \cite{AJ,Jit1, Jit3,LY}. The spectral behavior on the critical line
$|\lambda|=e^{\beta(\alpha)}$ was also investigated in~\cite{AJZ}. 
\item It was conjectured in \cite{JL2} that this sharp transition is universal for one-frequency quasiperiodic operators. It was proved for the Maryland model, $f(x)=\lambda\tan(\pi(x-\tfrac{1}{2}))$, in \cite{JL1,Simon1, HJY}, and for anti-Lipshitz monotone potentials \cite{JK2}. More recently, it was proved for all Type I operators \cite{GJ}, as well as in several other settings \cite{GYZ,HX}, see also \cite{KZ}. As mentioned, quantitative transversality is an important feature of all these families.

\item The absence of point spectrum in $\{E:L(E)<\beta(\alpha)\}$ is actually just a special case of the sharp Gordon-type result of \cite{JK3}, which applies to general functions $f$ of bounded variation. Therefore, we only need to prove the localization part of the result.

\end{enumerate}
\end{remark}
We emphasize that no quantitative lower bound on the derivative of $f$ is assumed: the derivative may vanish on a set of measure arbitrarily close to $1$, and $f$ is allowed to have arbitrarily large flat regions. The only requirement is that $f$ be non-constant and $C^1$. In particular, the continuous differentiability guarantees the existence of \emph{some} interval on which $f$ has positive derivative, and this alone turns out to suffice. To our knowledge, this is the first nonperturbative (and also the first sharp arithmetic) localization result for a class of quasiperiodic potentials which are allowed to have flat pieces.

The only exception known to us that definitely allows flat pieces is also in the monotone category, namely the perturbative results of \cite{KPS1,KKPS}, where localization at Diophantine $\alpha$ is established for a class of Maryland-type potentials:
$
\lim_{\varepsilon\to0^+}f(\varepsilon)=-\infty,
\qquad
\lim_{\varepsilon\to 0^+}f(1-\varepsilon)=\infty,
$
where $f$ may contain flat regions but require additional structural assumptions on their size and location.

In fact, one should note that most of the existing localization results for monotone quasiperiodic potentials (e.g. \cite{BLS,CSZ3,KKPS, KPS1,KPS2}) are \emph{perturbative} in nature,
and in particular do not yield a sharp arithmetic transition at the threshold $L(E)=\beta(\alpha)$. By contrast, our result establishes the full transition in the same sense as for the almost Mathieu and Maryland models or \cite{GJ,JK2}, without any frequency exclusions or perturbative assumptions. In particular, Theorem \ref{thm:main1} shows that flatness itself is not an obstruction: a potential may be essentially constant on most of the circle and still exhibit sharp arithmetic localization.

What is perhaps even more surprising is that even monotonicity is not required:

\begin{theorem}\label{thm:main2}
    Let $f$ be as in Theorem \ref{thm:main1}. Then there exists for every $\alpha\notin\mathbb{Q}$ some $\eta>0$ depending on $f$ and $\alpha$ such that for every bounded $C^1$ function $g:[0,1)\to\mathbb{R}$, if $|g'(x)|<\eta$ for all $x\in[0,1)$, then the quasiperiodic family associated with $f+g$ exhibits a sharp arithmetic transition.
\end{theorem}
	Starting with a $C^1$ function $f$ which is mostly flat, using Theorem \ref{thm:main2} one can find $g$ so that on the majority of the interval, $f+g$ is monotone decreasing. See Figure~\ref{fig:intro} for some visual examples. Sampling functions of this type can be considered to be perturbatively monotone. Although they need not be monotone themselves, the interplay with the kinetic part forces the eigenvalue curves of all sufficiently large finite-volume restrictions to be monotone on each monotonicity interval. Let us note that this result is not perturbative in the sense that (de-)localization of $f$ implies the same for $f+g$; whether $f+g$ is localized at some energy $E$ or not depends entirely on the Lyapunov exponent of $f+g$, not that of $f$.

We also mention that continuity assumptions on $f$ may be relaxed. Since our methods only require that, for each $x$, its trajectory will hit a point with large derivative in bounded time, we can establish the following.
\begin{theorem}\label{thm:main3}
    Let $f$ be a bounded, monotone, piecewise absolutely continuous function. Suppose that there is some interval $I\subseteq[0,1)$ and $\gamma>0$ such that for almost every $x\in I$, $f'(x)>\gamma$. Suppose in addition that the discontinuities of $f$ do not lie on the same $\alpha$-trajectory. Then $H(x)$ exhibits a sharp arithmetic transition.
\end{theorem}
The assumption on the location of the discontinuity points can be weakened substantially \cite{ReturnWords}, but with the present proof strategy it is not possible to remove it entirely. However, it is compatible with the Damanik-Killip \cite{DK} argument for positivity of the Lyapunov exponent (see also \cite{JK1}). 

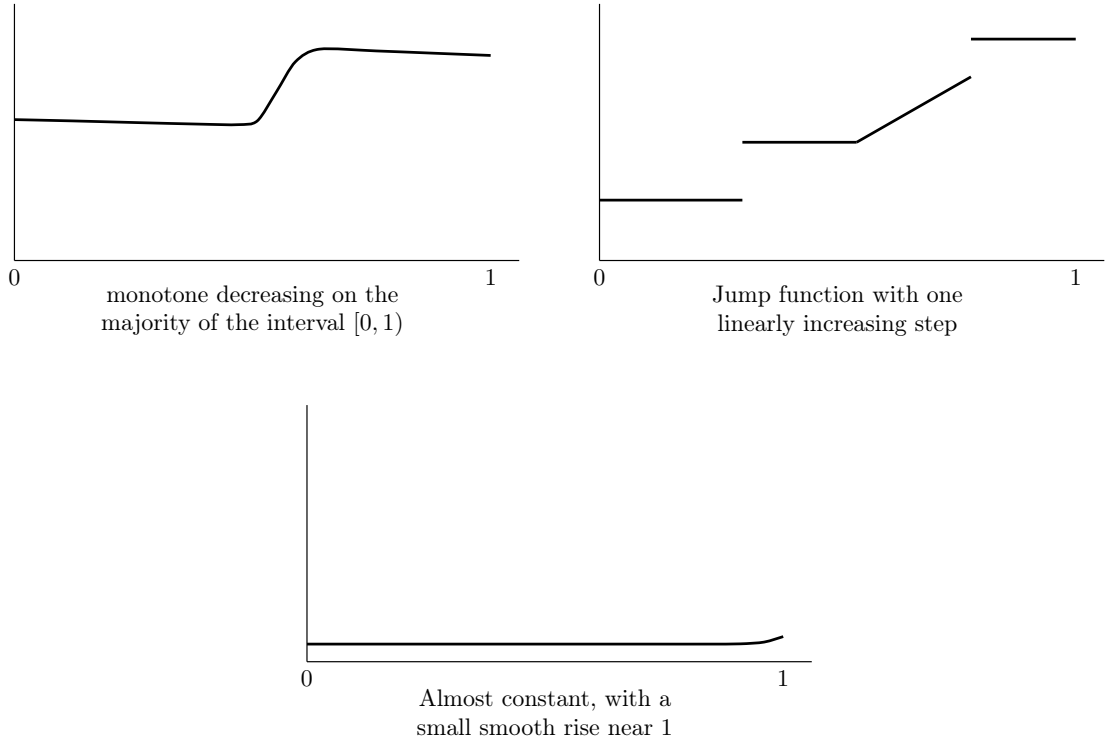
\begin{figure}[h!]
\centering
\scalebox{0.9}{
\begin{tikzpicture}[x=7.0cm,y=3.7cm]

\tikzset{
    myaxis/.style={black, line width=0.5pt},
    mygraph/.style={black, line width=1.15pt},
    descr/.style={font=\small, align=center, text width=4.9cm},
    ticklabel/.style={font=\small}
}

\begin{scope}[shift={(0cm,0cm)}]
    \draw[myaxis] (0,0) -- (1.06,0);
    \draw[myaxis] (0,0) -- (0,1.02);

    \node[ticklabel, below] at (0,0) {$0$};
    \node[ticklabel, below] at (1,0) {$1$};

    \draw[mygraph]
        plot[smooth] coordinates {
            (0.00,0.56)
            (0.12,0.555)
            (0.26,0.548)
            (0.39,0.542)
            (0.47,0.540)
            (0.51,0.555)
            (0.55,0.67)
            (0.59,0.79)
            (0.64,0.84)
            (0.76,0.833)
            (0.88,0.824)
            (1.00,0.815)
        };

    \node[descr] at (0.5,-0.20)
    {monotone decreasing on the majority of the interval $[0,1)$};
\end{scope}

\begin{scope}[shift={(8.6cm,0cm)}]
    \draw[myaxis] (0,0) -- (1.06,0);
    \draw[myaxis] (0,0) -- (0,1.02);

    \node[ticklabel, below] at (0,0) {$0$};
    \node[ticklabel, below] at (1,0) {$1$};

    \draw[mygraph] (0.00,0.24) -- (0.30,0.24);
    \draw[mygraph] (0.30,0.47) -- (0.54,0.47);
    \draw[mygraph] (0.54,0.47) -- (0.78,0.73);
    \draw[mygraph] (0.78,0.88) -- (1.00,0.88);

    \node[descr] at (0.5,-0.20)
    {Jump function with one linearly increasing step};
\end{scope}

 \begin{scope}[shift={(4.3cm,-5.9cm)}]
     \draw[myaxis] (0,0) -- (1.06,0);
     \draw[myaxis] (0,0) -- (0,1.02);

     \node[ticklabel, below] at (0,0) {$0$};
     \node[ticklabel, below] at (1,0) {$1$};

     \draw[mygraph]
         plot[smooth] coordinates {
             (0.00,0.07)
             (0.18,0.07)
             (0.36,0.07)
             (0.54,0.07)
             (0.72,0.07)
             (0.84,0.07)
             (0.91,0.071)
             (0.96,0.078)
             (1.00,0.10)
         };

     \node[descr] at (0.5,-0.20)
     {Almost constant, with a small smooth rise near \(1\)};
\end{scope}

\end{tikzpicture}}
\caption{Top row: Examples of sampling functions which satisfy the conditions of Theorem~\ref{thm:main2} or ~\ref{thm:main3} respectively. Bottom: For Theorem~\ref{thm:main1} the sampling function may have an arbitrarily large flat piece.}
\label{fig:intro}
\end{figure}

We  have an immediate consequence for Diophantine frequencies (where $\beta(\alpha)=0$):
\begin{corollary}\label{cor}
		For Diophantine $\alpha$ under the conditions of Theorem~\ref{thm:main3}, if $f$ is non-zero then the Lyapunov exponent $L(E)$ is positive for Lebesgue-almost every $E$ and therefore $H(x)$ almost surely has pure point spectrum with exponentially localized eigenfunctions.
\end{corollary}
Since there is no condition on the size of $f$ this generalizes the all-coupling localization of \cite{JK1} proven there for monotone anti-Lipschitz sampling functions.

\begin{remark}
It is instructive to compare this with the Sturmian case. Fix
an irrational $\alpha\in(0,1)$ with $\beta(\alpha)=0$ and $\lambda>0$,
and consider
\[
 f_0(x)=\lambda\mathbf 1_{[1-\alpha,1)}(x),\qquad x\in[0,1).
\]
The corresponding operators have purely singular continuous spectrum
for every phase~\cite{dkl}. Now replace the increasing jump at
$1-\alpha$ by a smooth monotone transition on an arbitrarily short
interval, keeping the downward jump at $0$, see Figure~\ref{fig:sturmian_comparison}. By
Corollary~\ref{cor}, the resulting operators have Anderson
localization for almost every phase. Thus smoothing just one of the
two jumps changes the almost-sure spectral type, even though the
sampling function still has two flat pieces covering all but an
arbitrarily small part of the circle.\footnote{The statement is for almost every rather than every phase because, theoretically, a zero measure set of phases can have spectral measures partially supported on the (measure zero) set of $E$ with $L(E)=0$ where we do not make any claim. }
\end{remark}
\begin{figure}[h]
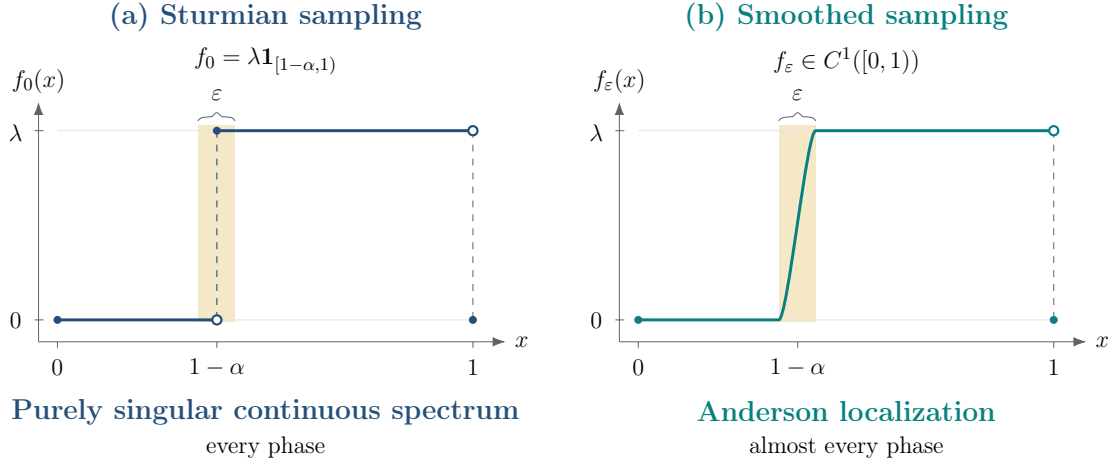

\include{figure_smoothing}
\caption{Comparison of the spectral types   between Sturmian potentials with $\beta(\alpha)=0$ and any monotone smoothing thereof.}
\label{fig:sturmian_comparison}
\end{figure}
\begin{remark}
In the context of Theorem~\ref{thm:main2} one does not get all-coupling localization if $f+g$ is non-monotone. The Lyapunov exponent is still positive for Lebesgue-almost every energy by \cite{DK} and therefore one has no absolutely continuous spectrum. However, if one adds a coupling parameter $H(x)=H_0 + \lambda (f+g)$ then the conditions of Theorem~\ref{thm:main2} themselves are only guaranteed to hold for $\lambda$ close to $1$. In Section~\ref{sec:ev_branches} we also establish a stricter but more quantitative criterion (Theorem~\ref{thm:bounded_implies_A1A2}) which allows the construction of non-monotone sampling functions which for fixed $\alpha$ satisfy the Anti-Lipschitz property for all small enough $\lambda> 0$. This gives, rather surprisingly, \emph{small-coupling} localization using our results. It seems likely that localization also should hold for large coupling or even all couplings, by virtue of positive Lyapunov exponent, but no applicable results seem to exist.
\end{remark}

\begin{remark}
Using our methods, it is not difficult to establish that there also exist unbounded non-monotone potentials for which the sharp arithmetic transition holds. For that reason it makes sense to develop the abstract theory for unbounded sampling functions (see Proposition \ref{prop:main} below), even though we will postpone specific constructions of examples to future work.
\end{remark}

\medskip

Let us now describe the approach. A central input in the work of Jitomirskaya--Kachkovskiy \cite{JK2} is the behavior of finite-volume eigenvalue curves.
Let $H_n(x)$ denote the restriction of $H(x)$ to $\{0,1,\dots,n-1\}$ (with Dirichlet boundary conditions), and write
\[
E_0(x)\le \cdots \le E_{n-1}(x)
\]
for its eigenvalues as functions of $x$. For each such curve, there will be $n$ negative jump discontinuities at the points $\{0,-\alpha,\ldots,-(n-1)\alpha\}$, which can be ordered as $\beta_0\leq\cdots\leq\beta_{n-1}$. Inside each of those \emph{monotonicity intervals} $(\beta_j,\beta_{j+1})$ the functions $E_k(x)$ are monotone increasing if $f$ is monotone. 

Let us formulate two conditions on the eigenvalue curves:
\begin{enumerate}
    \item[$\mathcal{A}_1$] There exists some $n_0\in \NM$ such that for every $n\geq n_0$, every $0\leq j,k\leq n-1$ and for every $x\leq y$ inside the same monotonicity interval $x,y\in (\beta_j, \beta_{j+1})$ one has $E_k(y)-E_k(x)\geq\gamma(y-x)$ for some $\gamma>0$ which does not depend on $j,k,n$.
    \item[$\mathcal{A}_2$] For every $1\leq k\leq n-1$ and $x\in[0,1)$, $E_k(x)\geq E_{k-1}(x-0)$.
\end{enumerate}
If the sampling function $f$ is anti-Lipschitz then the property $\mathcal{A}_1$ holds with $\gamma$ equal to the anti-Lipschitz constant of the sampling function itself. The property $\mathcal{A}_2$ is also evident for monotone $f$: it is the
interlacing condition at discontinuities as any such function satisfies $f(0^+)<f(1^-)$ due to which passing a point $\beta_j$ is equivalent to a negative rank-$1$ perturbation. However, as we will see there are much larger classes of sampling functions which satisfy $\mathcal{A}_1$ and $\mathcal{A}_2$.  In particular, $f$ itself need not be anti-Lipschitz and not even monotone, yet have the anti-Lipschitz property of  the associated eigenvalue curves. 

One of the main results of this paper is to demonstrate that the large deviation estimates and, ultimately, proof of localization in \cite{JK2} for anti-Lipschitz sampling functions can be adapted to use only the properties $\mathcal{A}_1$ and $\mathcal{A}_2$ without any other input from the sampling function $f$. Precisely, we will show the following:

\begin{proposition}\label{prop:main}
    Let $f:[0,1)\to\mathbb{R}$, $\alpha\in[0,1)\setminus\mathbb{Q}$. Suppose that for the corresponding family of Schr\"odinger operators $H(x)$ both $\mathcal{A}_1$ and $\mathcal{A}_2$ hold. Suppose in addition that
    \[
        \int_0^1\log(1+|f(x)|)\,dx<\infty.
    \]
    Then the set
    \[
    \{E\in\mathbb{R}:L(E)>\beta(\alpha)\}
    \]
    can support only pure point spectrum.
\end{proposition}

\begin{remark}
    \begin{enumerate}
        \item We will assume throughout that $f(0^+)<f(1^-)$ since in that case, $\mathcal{A}_1\implies\mathcal{A}_2$ (see \cite[Section 5]{JK1}).
        \item As in \cite{JK2}, Proposition \ref{prop:main} can be strengthened in the sense that one can show uniform localization, also leading to a so-called ``Area Law'', as discussed in \cite{PS}.
    \end{enumerate}
\end{remark}
For the proof of Proposition \ref{prop:main} one can follow closely the lines of \cite{JK2}, in that regard the main contribution of the present paper is to show that the required large deviation estimates can be derived using only the abstract properties $\mathcal{A}_1$ and $\mathcal{A}_2$ without any other assumptions on the sampling function. This Proposition then forms the first part of the proof of Theorems~\ref{thm:main1} to \ref{thm:main3}, the second part is to verify that the sampling functions as stated do in fact satisfy the anti-Lipschitz eigenvalue property $\mathcal{A}_1$ for some $\gamma>0$. 

The proof of Proposition~\ref{prop:main} follows the argument
of~\cite{JK2}, once the required large-deviation estimates have
been established under $\mathrm{A1}$ and $\mathrm{A2}$. The remaining
task is to prove $\mathrm{A1}$ for the sampling functions considered here.

Let us explain where the difficulty lies. The basic mechanism is to use the Hellmann--Feynman formula to express the derivative of each eigenvalue curve as a weighted average of the sampled derivatives (whenever they exist):
\[
E_k'(x)=\sum_{i=0}^{n-1} |\psi_k(i)|^2\, f'(x+i\alpha),
\]
where the weights are the $\ell^2$-masses of the corresponding normalized eigenfunction. One then partitions the sampling points into a \emph{good} set, where $f'$ is uniformly large and positive, and a \emph{bad} set, where $f'$ may be small or even negative. The main task is to show that an eigenfunction cannot place too much of its mass on the bad set, so that a definite portion of the weight must remain on good sites. If, moreover, the positivity of $f'$ on the good set dominates its possible smallness or negativity on the bad set, then the weighted average above is uniformly positive. This yields monotonicity, and in fact anti-Lipschitz behavior, for all eigenvalue branches.

If the good set contains an interval, then by basic principles of irrational rotation, every bad site will have a good site at bounded distance. Naively, one would like to use transfer matrix arguments to show that the weight an eigenfunction takes at a bad site is proportional to the weight it takes on the nearby good site. The problem with this argument is that lower bounds on transfer matrices do not directly translate into lower bounds on the mass carried by good sites. Indeed, transfer matrices compare two-component vectors of the form $(\psi(n),\psi(n-1))$, whereas the Hellmann-Feynman formula involves the one-site quantities $|\psi(n)|^2$. Thus, even if the norm of a transfer matrix image is well controlled, this does not by itself prevent the mass from being arranged in a way that avoids the good sites.

To overcome this, we shift from single sites to finite intervals whose boundary points are good. Roughly speaking, the key point is to show that there exist uniform constants $C_1,C_2>0$ such that for every $x\in [0,1)$, for every energy $E$ in a fixed compact interval $[E_1,E_2]$, and every point $m\in\mathbb{Z}$, one can find an interval $I\subseteq\mathbb{Z}$ whose boundary points are good, whose distance from $m$ is at most $C_1$, and such that the Green's function of the restriction of $H(x)$ to $I$ satisfies
\[
\|(H_I(x)-E)^{-1}\|<C_2.
\]
	Starting from $m$, one can then reach $I$ in uniformly bounded time, and once inside $I$, the Poisson formula relates the values of a solution to its values at the boundary of $I$. Since the boundary points are good, this gives the required mechanism for transferring weight from an arbitrary site to good points. As we show in Section~\ref{sec:mass_transfer}, this argument will for a monotone potential always result in a strictly positive lower bound for the derivative unless there exists an energy $E$ which is a common eigenvalue for all Dirichlet restrictions between good sites. Equivalently, one needs to show that the eigenvalue equation has no non-zero solution that vanishes on all good sites. 
    
    The exclusion of a common Dirichlet eigenvalue is the main difficulty.
    
Since there is no known unique continuation principle which can handle this situation, we proceed in the following way. Suppose that such an eigenvalue exists. On each continuity interval,
monotonicity implies that the corresponding eigenvector can be chosen
independent of the phase and must vanish at every coordinate where
the potential varies. Thus only finitely many values of $f$ on its
flat pieces are sampled by these eigenvectors. We can change $f$ at
the remaining points and reduce the problem to a sampling function
which is locally constant outside some interval $G$.

For such a function, the Dirichlet restrictions are described by the
return words to $G$ in a finite interval partition. We prove that,
when the endpoints of the partition lie on distinct $\alpha$-orbits
and $G$ is sufficiently small, these words cannot have a common
Dirichlet eigenvalue, for any assignment of real potential values to
the letters. In particular, no choice of the plateau values can
produce a common resonance. This is the {\it symbolic unique-continuation
principle} of Theorem~4.15. Its proof follows the return words as $G$
is enlarged. A common eigenvalue would be preserved under the
resulting trimming operations, until one reaches an empty word or
the one-letter case. The empty word has no eigenvalues, and the
one-letter case is ruled out by coprimality of the return times.

The mass estimate obtained in this way is stable under small
$C^1$ perturbations. The contribution from the good sites can then
also compensate for small negative derivatives elsewhere, which
gives Theorem~\ref{thm:main2}.

As stated above, in some cases we can even allow negative derivatives, as long as the regions with positive derivative compensate for that.

The rest of the paper is structured as follows. In Section \ref{sec:setting} we present notation and preliminaries which are required for Proposition \ref{prop:main}. Subsequently, in Section \ref{sec:loc_proof}, we prove Proposition \ref{prop:main}, with most of the technical details written in the appendix. Section \ref{sec:ev_branches} contains the main novelty: we prove that $\mathcal{A}_1$ and $\mathcal{A}_2$ hold for the broad class of sampling functions described above, including functions with flat pieces and, in some cases, regions of negative derivative.

\section{Setting and assumptions}\label{sec:setting}
\subsection{Notation}

For $x\in\mathbb{R}$, we write
\[
\|x\|:=\dist(x,\mathbb{Z})
\]
for the distance from $x$ to the nearest integer, and
\[
\{x\}:=x-\lfloor x\rfloor
\]
for the fractional part of $x$.

For integers $n_1\leq n_2$, we denote the integer interval
\[
[n_1,n_2]:=\{n_1,n_1+1,\dots,n_2\}.
\]
Given a Schr\"odinger operator $H$ on $\ell^2(\ZM)$, we denote by $H_{[n_1,n_2]}$ the restriction of $H$ to $\ell^2([n_1,n_2])$ with Dirichlet boundary conditions.

If $g:\mathbb{R}\to\mathbb{R}$ and $x\in\mathbb{R}$, we write
\[
g(x\pm 0):=\lim_{t\to x^\pm} g(t).
\]
We also set $\mathbb{N}_0:=\mathbb{N}\cup\{0\}=\{0,1,2,\dots\}$.

\medskip

Let $x\in[0,1)$. A non-zero function $\psi:\mathbb{Z}\to\mathbb{C}$ is called a \emph{generalized eigenfunction} (at energy $E$) of $H(x)$ if for all $m\in\mathbb{Z}$,
\begin{equation}\label{eq:gen_eig_eq}
\psi(m-1)+\psi(m+1)+f(x+m\alpha)\psi(m)=E\psi(m),
\end{equation}
	and there exists a constant $D>0$ such that for all $m\in\mathbb{Z}$,
\begin{equation}\label{eq:gen_eig_growth}
|\psi(m)|\leq D(1+|m|).
\end{equation}
If this is the case, then we say that $E$ is a \textit{generalized eigenvalue}. It is well known that the spectral measure of $H$ is supported on the set of generalized eigenvalues (see, e.g., \cite{Ber}).

For an irrational $\alpha$, we will denote by $\left(\frac{p_k}{q_k}\right)_{k=1}^\infty$ its rational approximants given by the continued fraction expansion, and we denote
\[
\beta(\alpha)\coloneqq\underset{n\to\infty}{\limsup}\,\frac{\ln q_{n+1}}{q_n}.
\]

\subsection{Integrated density of states and Lyapunov exponent}

\subsubsection{Integrated density of states}
We briefly recall the standard definitions; for a detailed discussion see, for example, \cite{CL,CFKS,PF}.

Define the \emph{density of states measure} $\nu$ by
\begin{equation}
\nu(B)=\int_0^1 \langle \delta_0,\,\chi_B(H(x))\,\delta_0\rangle\,dx
\end{equation}
for any Borel set $B$. Its distribution function
\begin{equation}
N(E)=\nu((-\infty,E]), \qquad E\in\mathbb{R},
\end{equation}
is called the \emph{Integrated Density of States (IDS)}. Equivalently,
\begin{equation}
N(E)=\lim_{n\to\infty}\frac{1}{n}\int_{0}^{1}
\#\{\text{eigenvalues of }H_n(x)\le E\}\,dx,
\end{equation}
and the limit exists for all $E\in\mathbb{R}$. We will show the following in the appendix.
\begin{lemma}\label{lemma:ac_IDS}
    Under assumptions $\mathcal{A}_1$ and $\mathcal{A}_2$, $N(\cdot)$ is Lipschitz continuous and satisfies
    \begin{equation}
    |N(E_2)-N(E_1)|\leq\gamma^{-1}|E_2-E_1|
\end{equation}
    
\end{lemma}
\begin{remark}
In particular, since $\supp\nu$ is equal to the almost sure spectrum of the family $H(x)$ (see e.g. \cite[Chapters 9,10]{CFKS}), this implies that the spectrum has positive Lebesgue measure. 
\end{remark}
\subsubsection*{Lyapunov Exponent}
For $E\in\mathbb{R}$, define the one-step transfer matrix and the $n$-step transfer matrix by
\begin{equation}
A(x,E)=\begin{pmatrix} E-f(x) & -1 \\ 1 & 0 \end{pmatrix},
\qquad
T_n(x,E)=A(x+(n-1)\alpha,E)\cdots A(x,E).
\end{equation}
The \emph{Lyapunov exponent} at energy $E$ is
\begin{equation}
L(E)=\lim_{n\to\infty}\frac{1}{n}\int_0^1 \log\|T_n(x,E)\|\,dx,
\end{equation}
where the limit exists by Kingman's subadditive ergodic theorem \cite{Steele}.

\subsection{Green's function and regular points}

Consider a general self-adjoint Schr\"odinger operator $H:\ell^2(\mathbb{Z})\to\ell^2(\mathbb{Z})$. Fix $E\in\mathbb{R}$ and consider a solution to the eigenvalue equation, namely $\psi:\mathbb{Z}\to\mathbb{R}$ such that
\begin{equation}\label{eq:ev}
    \psi(m-1)+\psi(m+1)+V(m)\psi(m)=E\psi(m)
\end{equation}
for all $m\in\mathbb{Z}$. Given $n_1,n_2\in\mathbb{Z}$, recall that $H_{[n_1,n_2]}$ is the restriction of $H$ to the interval $[n_1,n_2]$ with Dirichlet boundary conditions. Given $E\notin \sigma(H_{[n_1,n_2]})$ and $m\in[n_1,n_2]$, we have the Poisson formula:
\begin{equation}\label{eq:poisson}
\begin{aligned}
\psi(m)
&= -(H_{[n_1,n_2]}-E)^{-1}(n_1,m)\,\psi(n_1-1) \\
&\quad -(H_{[n_1,n_2]}-E)^{-1}(m,n_2)\,\psi(n_2+1),
\end{aligned}
\end{equation}

for every $\psi$ which satisfies \eqref{eq:ev}. In particular, exponential decay of \mbox{$(H_{[n_1,n_2]}-E)^{-1}(n_i,m)$} for $i=1,2$ combined with \eqref{eq:gen_eig_growth} can lead to exponential decay of $\psi$, given that $\dist(m,n_i)$ is large enough. This motivates the following definition.

\begin{definition}
    Fix $\sigma\in(0,1)$. Given $n\in\mathbb{N}$ and $\mu>0$, an integer $m\in\mathbb{Z}$ will be called \textit{$(\mu,n)$-regular} if there exists an interval $[n_1,n_2]\subseteq\mathbb{Z}$, called the \textit{regularity interval}, such that $n_2-n_1+1=n$, and
    \begin{enumerate}
        \item $m\in[n_1,n_2]$.\label{item1}
        \item $\dist(m,\{n_1,n_2\})\geq\sigma n$.\label{item2}
        \item $|(H_{[n_1,n_2]}-E)^{-1}(n_i,m)|\leq e^{-\mu|n_i-m|}$ for $i=1,2$.
    \end{enumerate}
    If \ref{item1} and \ref{item2} hold and for some solution $\psi$ we have
    \begin{equation}\label{eq:psi_regular}
    |\psi(m)| \le e^{-\mu|m-n_1|}\,|\psi(n_1-1)| + e^{-\mu|m-n_2|}\,|\psi(n_2+1)|,
    \end{equation}
    then we will say that $m$ is $(\psi,\mu,n)$-regular and that $[n_1,n_2]$ is a $(\psi,\mu)$-regularity interval for $m$.
\end{definition}
For the rest of this paper we fix $\sigma=\frac{1}{5}$ in the definition of regularity.

We will need a block expansion lemma which allows one to chain regularity intervals together, see e.g.\ \cite{FSW,JK2,JL3}. We will use the formulation from \cite{JK2}, as it is best suited to this setting.

\begin{definition}
\label{def:good_collection_psi}
Fix $E\in\mathbb{R}$ and let $\psi$ be a solution of \eqref{eq:ev}. Let $[N_1,N_2]\subseteq [N_1^-,N_2^+]$ be integer intervals and let $\mu>0$, $L\in\mathbb N$.
We say that $[N_1,N_2]$ admits a \textit{$(\psi,\mu,L)$-good collection} if for every
$m\in[N_1,N_2]\cap\ZM$ there exists $n_m\ge L$ and $\mu_m\ge \mu$ such that $m$ is
$(\psi,\mu_m,n_m)$-regular with a $(\psi,\mu)$-regularity interval $I_n\subseteq (N_1^-,N_2^+)$.
\end{definition}

	\begin{proposition}[{\cite[Remark~5.4]{JK2}}]
\label{prop:inner_point_estimate}
Let $\psi:\ZM\to\CM$ satisfy the eigenvalue equation \eqref{eq:ev}. Fix $\mu>0$ and integer
intervals $[N_1,N_2]\subseteq [N_1^-,N_2^+]$. Assume that $[N_1,N_2]$ admits a
$(\psi,\mu,L)$-good collection in the sense of Definition~\ref{def:good_collection_psi}.
Define boundary maximizers
\[
|\psi(M_1)|:=\max\{|\psi(m)|:\ m\in [N_1^-,N_1)\},\,\,\,\,
|\psi(M_2)|:=\max\{|\psi(m)|:\ m\in (N_2,N_2^+]\}.
\]
Then for all $m\in[N_1,N_2]\cap\ZM$,
\begin{equation}\label{eq:JK54_coarse}
|\psi(m)| \le e^{-\mu(1-o(1))|m-N_1|}\,|\psi(M_1)|
          + e^{-\mu(1-o(1))|m-N_2|}\,|\psi(M_2)|,
\end{equation}
where $o(1)\to 0$ as $L\to\infty$.
\end{proposition}

\subsection{The regularity lemmas}
	In this subsection we state regularity lemmas proved in \cite{JK2} for monotone anti-Lipschitz potentials. In \cite{JK2}, the monotonicity of \(f\) is used to establish the finite-volume eigenvalue branch structure and, in particular, $\mathcal{A}_1$ and $\mathcal{A}_2$. Once these assumptions are in place, monotonicity of \(f\) is no longer used.

	In the appendix we show that $\mathcal{A}_1$ and $\mathcal{A}_2$ indeed are enough to establish the crucial large deviation estimates which are as strong as those used in \cite{JK2} as the main ingredients of the regularity lemmas. The Lemmas themselves can then be proven in our setting using the same arguments verbatim. To avoid repetition, we do not reproduce those proofs themselves.

Let us fix $E\in\mathbb{R}$, some small $\tau >0$ and some large constant $C\gg 1$. Consider some denominator in the continued fraction expansion of $\alpha$ given by $q_k$ for some $k\in\mathbb{N}$. For $j\in\mathbb{Z}$, we define
\[
    R_j=[jq_k-\tau q_k,jq_k+\tau q_k].
\]
We say that a point $m\in\mathbb{Z}$ is \textit{resonant} if $m\in R_j$ for some $j\in\mathbb{Z}$. Otherwise, we say that $m$ is \textit{non-resonant}. Let us also consider the smallest $k_0\in\mathbb{N}_0$ for which $2q_{k-k_0}\leq\tau q_k$, and then the largest $s\in\mathbb{N}_0$ such that $2sq_{k-k_0}\leq\tau q_k$. Finally, let $s'=\lfloor s/10\rfloor$.

\begin{lemma}\label{prop:diophantine_transition}
    Let $\varepsilon>0$. For large enough $k$, if $q_{k+1}\leq q_k^C$, then every $m\in\mathbb{Z}$ which satisfies $|m|\in[\lfloor q_k/2\rfloor+1,q_{k+1}-\lfloor q_k/2\rfloor-1]$ is $(L(E)-\varepsilon,q_k)$-regular.
\end{lemma}

\begin{lemma}\label{lemma:diophantine_transition}
    Let $\varepsilon>0$. For large enough $k$, if $q_{k+1}\leq q_k^{C}$ and $s\leq q_{k-k_0}^C$, then every $m\in\mathbb{Z}$ which satisfies $100\tau q_k\leq |m|<(1-\tau)q_k$ is $(L(E)-\varepsilon,2sq_{k-k_0}-1)$-regular.
\end{lemma}

 \begin{lemma}\label{lemma:diophantine_transition_case_2}
     Let $\varepsilon>0$. For large enough $k$, if $q_{k+1}\leq q_k^C$ and $s>q_{k-k_0}^C$, then every $m\in\mathbb{Z}$ which satisfies $100\tau q_k<|m|<q_k-\tau q_k$ is $(L(E)-\varepsilon,2s'q_{k-k_0}-1)$-regular.
 \end{lemma}

\begin{lemma}\label{lemma:Liouville_transition_nonresonant}
    Let $\varepsilon>0$. For large enough $k$, if $q_{k+1}>q_k^C$, then for every non-resonant point $m\in[-q_{k+1}/10,q_{k+1}/10]\cap[(j-1)q_k+\tau q_k,jq_k-\tau q_k]$ there exists $n\leq 2sq_{k-k_0}-1$ such that $m$ is $(L(E)-\varepsilon,n)$-regular, with $n\to\infty$ as $k\to\infty$.
\end{lemma}

\begin{lemma}\label{lemma:Liouville_transition_resonant}
Let $\psi$ be a generalized eigenfunction at energy $E$ normalized to $|\psi(0)|^2+|\psi(1)|^2=1$. For sufficiently large $k$, if $q_{k+1}>q_k^C$, then
every resonant point $m\in R_j$ with $0< |j|\leq\frac{q_{k+1}}{10q_k}$ is $\left(\psi,\frac{\log|j|}{q_k}+L(E)-\beta_k-3\tau,r(1+O(\tau))\right)$-regular, where $\beta_k=\frac{\ln q_{k+1}}{q_k}$ and $r=2q_k$ or $r=4q_k$.
\end{lemma}

\section{Proof of \texorpdfstring{Proposition~\ref{prop:main}}{Proposition 1.8}} \label{sec:loc_proof}
We now prove Proposition \ref{prop:main} following the proof of \cite[Theorem 1.1]{JK2}.

Fix $x\in[0,1)$, $E\in\mathbb{R}$. Suppose that $\psi:\mathbb{Z}\to\mathbb{R}$ satisfies \eqref{eq:gen_eig_eq} and \eqref{eq:gen_eig_growth}, namely $\psi$ is a generalized eigenfunction of $H(x)$ with generalized eigenvalue $E$. We consider nonnegative $0\leq m\in\mathbb{Z}$, the proof for negative $m$ is essentially identical. Let $\varepsilon>0$ be so small that $$\varepsilon<\frac{1}{4}(L(E)-\beta(\alpha))$$ and let $\tau>0$ be small enough for the regularity lemmas. In particular, we impose $$\tau <\min\left\{\frac{1}{24}\varepsilon,10^{-4}\right\}.$$ Set $\varepsilon'=\varepsilon/4$ and to absorb some other errors later assume further $100 L(E)\sqrt{\tau}<\varepsilon'$. Suppose that
\begin{equation}\label{eq:k_location}
    \sqrt{\tau} q_k\leq m< \sqrt{\tau} q_{k+1}
\end{equation}
which defines the scale $q_k$. We will now use the corresponding regularity lemmas depending on whether the scale is Diophantine or not.

	\textbf{Case 1: $q_{k+1}\leq q_k^C$}. We will show that in this case, for every $\varepsilon>0$, we have
\begin{equation}\label{eq:diophantine_decay}
    |\psi(m)|\leq e^{-(L(E)-\varepsilon)m}
\end{equation}
uniformly across all large enough $m$ satisfying \eqref{eq:k_location} with $q_{k+1}\leq q_k^C$.

Set
\[
N_1=\lfloor100\tau q_k\rfloor+1,\qquad
N_2=q_{k+1}-\lfloor q_k/2\rfloor-1.
\]
For integers $m'\in[N_1,(1-\tau)q_k)$, we apply
Lemma~\ref{lemma:diophantine_transition} if
$s\le q_{k-k_0}^C$, and
Lemma~\ref{lemma:diophantine_transition_case_2} otherwise.
In either case, $m'$ is
$(L(E)-\varepsilon',n_{m'})$-regular, where
$n_{m'}\le q_k$ and the lengths tend uniformly to infinity
as $k\to\infty$. 

For integers $m'\in[\lfloor q_k/2\rfloor+1,N_2]$ we apply 
Lemma~\ref{prop:diophantine_transition} to get $(L(E)-\varepsilon',q_k)$-regularity. As those two ranges overlap, every integer
$m'\in[N_1,N_2]$ is
$(L(E)-\varepsilon',n_{m'})$-regular, with lower bounds $L_k$ on the length satisfying
\[
L_k\le n_{m'}\le q_k,\qquad L_k\to\infty.
\]
The regularity intervals including their exterior boundary sites
are contained in
\[
[N_1^-,N_2^+]=[-q_k-1,q_{k+1}+q_k+1].
\]

Hence, by Proposition \ref{prop:inner_point_estimate}, we have for large enough $k$ 
\begin{equation}
    |\psi(m)|\leq e^{-(L(E)-2\varepsilon')|m-N_1|}|\psi(M_1)|+e^{-(L(E)-2\varepsilon')|m-N_2|}\, |\psi(M_2)|,
\end{equation}
with the maximizers $M_1,M_2$ as in Proposition~\ref{prop:inner_point_estimate} which we can bound
    \begin{align*}
        |\psi(M_1)| &\leq D(q_k+2) \leq D(\frac{m}{\sqrt{\tau}}+2),\\
        |\psi(M_2)| &\leq D(q_{k+1}+q_k+2)\leq D(\frac{m^C}{\sqrt{\tau}^C}+\frac{m}{\sqrt{\tau}}+2).
    \end{align*}    
    In addition, by the assumption on $\tau$, we have
    \[
        |m-N_1|\geq(1-100\sqrt{\tau})m-1,\qquad |m-N_2|= N_2-m \geq m
    \]
with the latter because $N_2-2m\geq (\frac12-2\sqrt{\tau})q_k-1 > 0$.
    Hence, we obtain for large enough $k$
    \begin{equation}
        |\psi(m)|\leq e^{-(L(E)-2\varepsilon')((1-100\sqrt\tau)m-1)} |\psi(M_1)|+e^{-(L(E)-2\varepsilon')m} |\psi(M_2)|.
    \end{equation}
    Using the fact that $\varepsilon'=\varepsilon/4$, $\tau\leq\varepsilon/24$ and that $\psi(M_1), \psi(M_2)$ are polynomially bounded in $m$, we obtain \eqref{eq:diophantine_decay} for large enough $m$ as required.

        \textbf{Case 2: $q_{k+1}>q_k^C$.} Here we make use of Lemmas \ref{lemma:Liouville_transition_nonresonant} and \ref{lemma:Liouville_transition_resonant} to show that
    \begin{equation}\label{eq:liouville_decay}
        |\psi(m)|\leq e^{-(L(E)-\beta(\alpha)-3\varepsilon)m}.
    \end{equation}
    Indeed, by these lemmas, there are $N_1,N_2$ such that every $m'\in [N_1,N_2]$ is $(\psi,\mu,n)$-regular with $\mu\geq L(E)-\beta_k-3\tau$ and $n\leq10q_k$ for large enough $k$. In particular, setting $$[N_1,N_2]=[ \lfloor\tau q_k\rfloor+1,\lfloor q_{k+1}/10\rfloor-q_k]$$ and $[N_1^-,N_2^+]=[-10q_k-1,q_{k+1}]$, we obtain
    \[
        |\psi(m)|\leq e^{-(L(E)-\beta_k-3\tau-\epsilon')|m-N_1|}|\psi(M_1)|+e^{-(L(E)-\beta_k-3\tau-\epsilon')|m-N_2|}|\psi(M_2)|
    \]
    similarly to the first case, with $M_1,M_2$ the boundary maximizers in $[N_1^-,N_1)$ and in $(N_2,N_2^+]$ respectively. Again, choosing $\tau$ correctly, we have that
    \[
        |m-N_1|\geq(1-\sqrt\tau)m-1
\]
and 
\[\left|m-N_2\right|\geq \left(\frac{1}{10}-\sqrt{\tau}\right)q_{k+1}-q_k-1 \geq \frac{q_{k+1}}{20} \geq \frac{m}{20\sqrt{\tau}} \geq 2m.
    \]
    By our choice of $\tau$ we can absorb the error into $\varepsilon$, this leads to
    \[
        |\psi(m)|\leq e^{-(L(E)-\beta_k-2\varepsilon)m},
    \]
    for large $m$ and since eventually $\beta_k\leq \beta(\alpha)+\epsilon$ this finishes the proof of \eqref{eq:liouville_decay}. \qed

	The following corollary states that one can obtain uniform localization, given some uniform positivity assumptions on the Lyapunov exponent. The proof is identical to the proof of \cite[Corollary 5.14]{JK2}; we include it here for completeness.
    \begin{corollary}
        Let $f$ satisfy the assumptions of Proposition \ref{prop:main}. Then for every $\delta>0$, there exists a sequence $\varepsilon_n\to0$, depending on $f,\alpha$ and $\delta$, such that for every $E\in\mathbb{R}$, if $L(E)>\beta(\alpha)+\delta$, then for every generalized eigenfunction $\psi$ corresponding to $E$ which satisfies $\|\psi\|_\infty\leq1$, there exists $n_0=n_0(\psi)$ such that for every $n\in\mathbb{Z}$,
        \[
            |\psi(n-n_0)|\leq e^{-(L(E)-\beta(\alpha))(1-\varepsilon_n)|n-n_0|}.
        \]
    \end{corollary}
    \begin{proof}
        In the course of the proof of Proposition \ref{prop:main}, the only step which is not uniform in $E$ is when we choose the constant $D$ in \eqref{eq:gen_eig_growth}. By our assumption, $\|\psi\|_\infty\leq 1$ and one may shift the origin so that $\psi$ attains its maximum at the origin; then one can take $D=1$.
    \end{proof}

    \begin{corollary}
        Let $f$ satisfy the assumptions of Proposition \ref{prop:main}. Suppose that $A\subseteq\mathbb{R}$ is such that for almost every $E\in A$, $L(E)>\beta(\alpha)$. Then for almost every $x\in[0,1)$, $H(x)$ has only pure point spectrum inside $A$.
    \end{corollary}
    \begin{proof}
        By Lemma \ref{lemma:ac_IDS}, the density of states is absolutely continuous w.r.t.\ the Lebesgue measure, and so 
        \[
        \nu\left(\left\{E\in A:L(E)\leq\beta(\alpha)\right\}\right)=0.
        \]
        This implies that the same is true for the spectral measure of $H(x)$ for Lebesgue-almost every $x\in[0,1)$, which implies the result.
    \end{proof}
\section{Anti-Lipschitz eigenvalue branches}\label{sec:ev_branches}

	In this section we prove that there are potentials for which for all sufficiently large volumes, all eigenvalue branches are monotone and anti-Lipschitz on each monotonicity interval, even though the potential is not itself anti-Lipschitz or even monotone.

\newcommand{\Aset}{\mathsf{A}}    
\newcommand{\NA}{\mathcal N_{\Aset}}

\begin{theorem}\label{thm:bounded_implies_A1A2}
Fix an irrational $\alpha\in\TM$, a constant $C>0$ and a measurable set $\Aset\subset\TM$ such that $\Aset\cap (\Aset -\alpha)$ contains an interval. Then there exists a constant $c_{\Aset, \alpha, C} >0$ such that the following holds:

Assume that the sampling function $f:\mathbb{T} \to\mathbb{R}$ is bounded by $\lVert f\rVert_\infty \leq C$ and that $f$ is either monotone and piecewise continuous, or that $f$ is piecewise absolutely continuous and has only positive jumps. Then Assumptions $\mathcal{A}_1$ and $\mathcal{A}_2$ are satisfied if
$$|\gamma_-| < c_{\Aset, \alpha, C} \gamma_+$$
with
$$\gamma_+:=\essinf_{x\in\Aset} f'(x), \qquad \gamma_-:= \essinf_{x\notin\Aset} f'(x).$$
\end{theorem}
The positive jumps here refer to interior jumps in $(0,1)$, not the reset at the end of the interval. We adopt the convention that $f$ is right-continuous. Piecewise absolutely continuous here precisely means that $f$ has at most finitely many jump discontinuities and is absolutely continuous in between those jumps. All those jumps are assumed to be positive, which is automatic if $f$ is monotone. From the proof one can in principle also derive an explicit lower bound for the constant $c_{\Aset, \alpha, C}$ depending on the size of the interval in $\Aset\cap (\Aset -\alpha)$ and the continued fraction expansion, but we will refrain from doing so since it would be far from optimal.

For purely monotone potentials we can prove a much more general result albeit at the cost of a less explicit bound:
\begin{theorem}\label{thm:monotonederivative_implies_A1A2}
Fix an irrational $\alpha\in\TM$ and consider a sampling function $f:\mathbb{T} \to\mathbb{R}$ which is bounded, monotone and piecewise continuous on $(0,1)$ and has some interval $I\subset \mathbb{T}$ with the almost everywhere existing derivative satisfying $\gamma_+:=\essinf_{x\in I} f'(x)>0$. If no two of the discontinuity points of $f$ lie on the same $\alpha$-orbit then the Assumptions $\mathcal{A}_1$ and $\mathcal{A}_2$ are satisfied.
\end{theorem}
In particular, if $f$ is a $C^1$-function then this holds whenever $f$ is not constant. The assumption on the location of the discontinuity points can be significantly weakened, but it cannot be removed entirely. One can in fact construct examples of monotone potentials which satisfy the conditions of Theorem~\ref{thm:monotonederivative_implies_A1A2} except that their discontinuities are on the same $\alpha$-orbit and which violate Assumption $\mathcal{A}_1$ as infinitely many Dirichlet restrictions have some eigenvalue curves which are not Anti-Lipschitz. Details will be published elsewhere. 

The strategy of proof allows to incorporate slightly negative derivatives perturbatively, though with a similarly inexplicit constant:
\begin{theorem}
\label{thm:monotonederivative_plus_perturbation_implies_A1A2}
Let $f$ satisfy the conditions of Theorem~\ref{thm:monotonederivative_implies_A1A2} and let $f$ in addition be piecewise absolutely continuous. Then there exists for every $\alpha \notin \mathbb{Q}$ some constant $C>0$ such that for every piecewise absolutely continuous function $g$ with only positive jumps, $\|g\|_\infty \leq C$ and $\|g'\|_\infty \leq C$ the sampling function $\tilde{f}=f+g$ satisfies Assumptions $\mathcal{A}_1$ and $\mathcal{A}_2$.
\end{theorem}
If $g$ is a $C^1$-function then $g$ is the integral of $g'$ up to an irrelevant constant shift and therefore also bounds $\|g\|_\infty$, so Theorem~\ref{thm:monotonederivative_plus_perturbation_implies_A1A2} implies in particular Theorem~\ref{thm:main2}.

\begin{remark}
We can prove existence of $\eta$, but it is difficult to express a lower bound explicitly in terms of the sampling function as it depends on $\alpha$ and the eigenvalues of some sufficiently large finite-size restrictions.  In contrast to Theorem~\ref{thm:bounded_implies_A1A2} one should note that $\eta$ depends explicitly on the function $f$ which is perturbed whereas the constant $c_{\Aset, \alpha, C}$ only depends on the location and size of the regions with positive or negative derivative, but not on the sampling function itself.
\end{remark}

\subsection{Warm-up}
\label{sec:warmup}
We first explain the core idea in the simplest case. Throughout, let $H_n(x)$ be the restriction of the Schrödinger operator 
$$(H(x)\psi)(i)=\psi(i+1)+\psi(i-1)+f(x+i\alpha)\,\psi(i), \qquad i\in \mathbb{Z}$$ to the interval $[0,n-1]$. For simplicity we assume $f$ is continuously differentiable on $(0,1)$. Since each $H_n(x)$ has exactly $n$ simple eigenvalues we can then label them uniquely as piecewise continuous curves $x\mapsto E_k(x)$ for $k=1,...,n$ which are differentiable anywhere outside of $n$ jump points $\beta_1,...,\beta_n$.

Let $(E_k(x), \psi_k)$ be a pair of an eigenvalue and an associated normalized eigenfunction. The derivative of the $k$-th eigenvalue curve with respect to $x$ is given by the Hellmann-Feynman formula,
\begin{equation}\label{eq:warmup_HF}
E_k'(x)=\sum_{i=0}^{n-1} |\psi_k(i)|^2\, f'(x+i\alpha)
\end{equation}
as we will also rederive further below. The goal is to prove a deterministic strictly positive lower bound on the right hand side, which then immediately proves the anti-Lipschitz property for the eigenvalue curves. This would be trivial if $f'$ had a strictly positive lower bound, however, we will instead only want to assume that there is a (large) \emph{good} subset $\Aset \subset \mathbb{T}$ where the derivative has a strictly positive lower bound and conversely a (small) \emph{bad} subset $\Aset^c$ where the derivative may tend to zero or even become negative. What we can use to compensate is that sites that sample the bad subset, i.e. sites $i\in \mathbb{Z}$ with $x+i\alpha \in \Aset^c$, are often fairly sparse. In particular, it is often the case that if $i\in \mathbb{Z}$ is a bad site with small derivative $f'(x+i\alpha)$ then the derivative will be much larger at all (next-)nearest neighbors. In that case the eigenvalue equation (or equivalently, the transfer matrix approach) guarantees that no eigenfunction can be supported exclusively on bad sites. We can turn this easily into a quantitative estimate:

\begin{lemma}
\label{lemma:mass_estimate_warmup}
Consider a partition of the integer interval $[0,n-1]$ into two disjoint subsets $G$, $B$. Assume that for every $i\in B$ the sites $\{i-2,i-1, i+1,i+2\} \cap [0,n-1]$ all lie in $G$. 

If $(E, \psi)$ is an eigenvalue-eigenfunction pair of $H_n(x)$ for some $x\in \mathbb{T}$ then
$$\sum_{i\in G} |\psi(i)|^2 \geq \eta \sum_{i\in B} |\psi(i)|^2$$
with the constant $\eta=\frac{1}{2(3+2\lVert f\rVert_\infty)^2}$. This implies
$$\sum_{i\in G} |\psi(i)|^2\geq \frac{\eta}{1+\eta}$$
and
$$\sum_{i\in B} |\psi(i)|^2\leq \frac{1}{1+\eta}.$$
\end{lemma}
\begin{proof}
We use the abbreviation $a_i= E-f(x+i\alpha)$ with which the eigenvalue equation is rearranged to
\begin{equation}\label{eq:warmup_recursion}
a_{i+1}\psi(i+1)-\psi(i+2)=\psi(i)
\end{equation}
for $i< n-2$. As \begin{align*}
|\psi(i)|^2 &\leq |a_{i+1}\psi(i+1)|^2 + |\psi(i+2)|^2 + 2|a_{i+1}\psi(i+1)\psi(i+2)| \\
&\leq (|a_{i+1}|+1)^2 \max(|\psi(i+1)|^2, |\psi(i+2)|^2)
\end{align*}
we have 
$$\max(|\psi(i+1)|^2, |\psi(i+2)|^2) \geq \frac{1}{(|a_{i+1}|+1)^2} |\psi(i)|^2\geq \tilde{\eta} |\psi(i)|^2$$
with $\tilde{\eta}=\frac{1}{(3+2\lVert f\rVert_\infty)^2}$ where we use that $E\in \sigma(H_n(x))\subset [-2-\lVert f\rVert_\infty, 2+\lVert f\rVert_\infty]$.

If $i\in B$ then by assumption both sites on the left-hand side are in $G$ and at least on one of those \emph{good} sites the eigenfunction $\psi$ has mass comparable to that on $i$. To cover the case when $i$ is at the upper end of the interval one has to use the recursion in the opposite direction to show also
$$\max(|\psi(i-1)|^2, |\psi(i-2)|^2) \geq \frac{1}{(|a_{i-1}|+1)^2} |\psi(i)|^2\geq \tilde{\eta} |\psi(i)|^2$$
again with both sites on the left-hand side in $G$ if $i\in B$.

In conclusion, for every \emph{bad} site $i\in B$ there exists a nearby \emph{good site} $j\in G$ with $|i-j|\leq 2$ such that
$$|\psi(j)|^2\geq \tilde{\eta} |\psi(i)|^2.$$
Summing over $i\in B$ and accounting for the fact that the same $j\in G$ can be produced by at most two different sites in $B$ we conclude
$$2 \sum_{i\in G} |\psi(i)|^2 \geq \tilde{\eta} \sum_{i\in B} |\psi(i)|^2.$$
The remaining inequalities are immediate consequences of $\ell^2$-normalization.
\end{proof}

We now consider a simple example where this strategy applies:
\begin{proposition}
\label{proposition:hf_bound_warmup}
We partition $\mathbb{T}$ into the interval $\Aset = (0,1-\delta)$ and its complement $\Aset^c = [1-\delta,1]$. Assume that $\delta < \min(\lVert \alpha\rVert, \lVert 2\alpha\rVert)$. 

Let $f:\mathbb{T}\to \mathbb{R}$ be continuously differentiable on $(0,1)$. We define
$$\gamma_+ = \inf_{x\in\Aset} f'(x), \qquad \gamma_- = \inf_{x\in \Aset^c} f'(x),$$
and assume that $\gamma_+>0$ and $\gamma_-\leq 0$. Then any eigenvalue curve satisfies
\begin{equation}
\label{eq:hf_bound_warmup}
E_k'(x) \geq c_+ \gamma_+ + c_- \gamma_-
\end{equation}
with the constants
$$c_+ = \frac{1}{1+2(3+2\lVert f\rVert_\infty)^2}, \qquad c_- = 1-c_+ = \frac{2(3+2\lVert f\rVert_\infty)^2}{1+2(3+2\lVert f\rVert_\infty)^2}$$
and for any $x \in (0,1)$ not equal to one of the jump points $\beta_1,...,\beta_n$. 
\end{proposition}
\begin{proof}
For fixed $x\in (0,1)$ set $G=\{i \in [0,n-1]: x+i\alpha\in \Aset\}$ and $B=[0,n-1]\setminus G$. By assumption on $\delta$ the sets satisfy the condition of Lemma~\ref{lemma:mass_estimate_warmup} and hence the Hellmann-Feynman formula \eqref{eq:warmup_HF} implies
$$E_k'(x)=\sum_{i\in G} |\psi_k(i)|^2\, f'(x+i\alpha) +  \sum_{i\in B} |\psi_k(i)|^2\, f'(x+i\alpha) \geq \frac{\eta}{1+\eta}\gamma_+ + \frac{1}{1+\eta}\gamma_-.$$
\end{proof}

\begin{remark}
Note that for a Schrödinger-type operator of the form
$$H = \Delta + \lambda f$$
the bound \eqref{eq:hf_bound_warmup} depends on the coupling constant $\lambda$; it becomes better with smaller $\lambda$, eventually tending to the limit
$$c_+ \to \frac{1}{19}, \qquad c_- \to \frac{18}{19}.$$ 
Hence, given a potential with $\frac{|\gamma_-|}{\gamma_+} < \frac{1}{19}$ the conditions of the Proposition will be satisfied at least for some small enough $\lambda>0$.
\end{remark}

In general, we run into two problems when generalizing this simple argument. For one, the set with bad derivatives may not be small enough to satisfy the covering condition of Proposition~\ref{proposition:hf_bound_warmup} and one may need to move farther away from a bad site to be guaranteed to find a good site. The more challenging problem is that we may then also fail to find \emph{consecutive} good sites and therefore the simple argument based on the recurrence relation alone will not work. What one then needs is some form of a unique continuation principle which guarantees that no solution can be simultaneously small on all sites with good derivatives.

\subsection{The Hellmann-Feynman formula}

In the general case we want to still use the Hellmann-Feynman formula but without imposing continuous differentiability. As we eventually want to integrate the derivatives to get information about the eigenvalue curves, it is enough to have differentiability almost everywhere. This is in particular implied by (piecewise) absolute continuity or monotonicity. Let us provide a precise formulation and proof to make sure that it does not require any hidden continuity assumptions on the derivatives:

\begin{proposition}[Hellmann-Feynman]
\label{prop:HF_formula}
Let $x\in I \mapsto h(x)\in M_n(\mathbb{C})$ be a continuous function mapping an interval to self-adjoint matrices with simple eigenvalues. For every $x$ there exists a neighborhood $U$ such that there exist $n$ orthonormal eigenvalue-eigenvector pairs $(E_k(y),\psi_k(y))$ for $y\in U$ such that $y\in U\mapsto E_k(y)$ and $y\in U\mapsto \psi_k(y)$ are continuous. If $h$ is differentiable at $x$ then in addition the functions $E_k$ are differentiable at $x$ and
\begin{equation}\label{eq:HF}
E'_k(x)=  \frac{\langle\psi_k(x), h'(x)\psi_k(x)\rangle}{\langle\psi_k(x),\psi_k(x)\rangle}
\end{equation}
If $h$ is absolutely continuous in the neighborhood $U$ then the same is true for the functions $E_k$ and
$$E_k(x)-E_k(y) = \int_{y}^{x} E'_k(z) dz$$
for all $x,y\in U$ with the almost everywhere defined derivatives.
\end{proposition}
\begin{proof}
Since each $h(y)$ has $n$ simple (and hence isolated) eigenvalues the (absolute) continuity of the eigenvalue curves $E_k$ follows from the (absolute) continuity of $h$ e.g. through the Weyl inequality
$$|E_k(x)-E_k(y)|\leq \|h(x)-h(y)\|.$$
Since there are no eigenvalue-crossings it is a simple consequence of holomorphic functional calculus that the spectral projections onto the $k$-th eigenvalue $P_k$ are also continuous. Similarly, if $h$ is differentiable at $x$ then $P_k$ is also operator-norm-differentiable at $x$. Therefore the eigenvalue functions $E_k(y)$ defined by $P_k(y) h(y)= E_k(y) P_k(y)$ are also continuous respectively differentiable. Choosing any non-trivial pair $(E_k(x), \phi_k(x))$ the eigenvectors $$\psi_k(y)=\frac{P_k(y)\phi_k(x)}{\|P_k(y)\phi_k(x)\|}$$ are well-defined and continuous on a sufficiently small neighborhood $U$ of $x$ and differentiable at $x$. The Hellmann-Feynman formula \eqref{eq:HF} is then a simple consequence of the Leibniz rule applied to 
$$\frac{d}{dx} E_k(x)= \frac{d}{dx} \langle \psi_k(x), h(x)\psi_k(x)\rangle$$
and noting that two of the three resulting terms cancel due to 
$$0=\frac{d}{dx} 1= \frac{d}{dx}\langle \psi_k(x),\psi_k(x)\rangle = \langle \psi'_k(x),\psi_k(x)\rangle + \langle \psi_k(x),\psi'_k(x)\rangle. $$

\end{proof}

For the finite-volume Schrödinger-type Hamiltonian $H_n(x)$ on the integer interval $[0,n-1]$ with the sampling function $f$ and Dirichlet boundary conditions the identity for the derivative simplifies as follows from \eqref{eq:HF}:

\begin{corollary}
\label{cor:HF_formula}
If $f\rvert_{(0,1)}$ is a piecewise absolutely continuous sampling function then for all $n$ all eigenvalue branches $E_k(x)$ of $H_n(x)$ are piecewise absolutely continuous.

		If $f\rvert_{(0,1)}$ is instead monotone and piecewise continuous then all eigenvalue branches $E_k(x)$ of $H_n(x)$ are monotone and piecewise continuous on each monotonicity interval.

		In both cases the eigenvalue curves $E_k$ are almost everywhere differentiable on each monotonicity interval. If $(E_k(x), \psi_k)$ is an eigenvalue-eigenvector pair of $H_n(x)$ with normalized $\|\psi_k\|=1$ and $f$ is differentiable at every sampled point $x+i\alpha$, $0\leq i<n$, then one has
\begin{equation}\label{eq:HF2}
E'_k(x)=\sum_{i=0}^{n-1} |\psi_k(i)|^2\,f'(x+i\alpha)
\end{equation}
\end{corollary}

\subsection{Mass transfer}
\label{sec:mass_transfer}
In this section we prove Theorem~\ref{thm:bounded_implies_A1A2} based on a generalization of Lemma~\ref{lemma:mass_estimate_warmup} as well as lay the groundwork for the proof of Theorem~\ref{thm:monotonederivative_implies_A1A2}. The main input is the following simple observation:
\begin{lemma}
\label{lemma:poisson}
Let \(H\) be a discrete Schrödinger operator on an integer interval containing $[m_1, m_2]$ for some $m_1 < m_2$ with $|m_1-m_2|>1$ and let \(\psi\) solve \(H\psi=E\psi\). Denote by \(H_{(m_1,m_2)}\) the Dirichlet restriction of \(H\) to the finite interval \( (m_1,m_2)=[m_1+1,m_2-1]\) (i.e.\ the finite matrix with Dirichlet boundary conditions at \(m_1\) and \(m_2\)).
Assume
\[
\mathrm{dist}\!\big(E,\sigma(H_{(m_1,m_2)})\big)\;=\;\delta>0.
\]
Then for every \(m\in(m_1,m_2)\) we have 
\[
|\psi(m)| \;\le\; \frac{1}{\delta}\big(|\psi(m_1)| + |\psi(m_2)|\big).
\]
In particular
\[
\max\{|\psi(m_1)|,|\psi(m_2)|\} \;\ge\; \frac{\delta}{2}\,|\psi(m)|.
\]
\end{lemma}
\begin{proof}
This is an immediate consequence of the Poisson formula. 
\end{proof}

This is a very basic unique continuation principle: if $E$ is not an eigenvalue of $H_{(m_1,m_2)}$ then any solution of the eigenvalue equation can only vanish on both boundary points $m_1, m_2$ if it vanishes identically. 

We now prove that under certain conditions if we partition $\mathbb{T}$ into a set $\Aset$ where the sampling function has positive derivatives and its complement, then eigenfunctions should have at least a fraction of their mass concentrated on the sites associated to $\Aset$.
\begin{lemma}
\label{lemma:poisson_mass_transfer}
Consider the finite-volume Hamiltonian $H_n(x)$ on the integer interval $[0,n-1]$ with a bounded sampling function $f$ such that $\lVert f\rVert_\infty \leq C$ for some constant. Let $(E,\psi)$ be any eigenvalue-eigenvector pair of $H_n(x)$ for some $x\in \mathbb{T}$ with $\psi$ normalized. 

		Assume that there exists some set $\Aset \subset \mathbb{T}$, an open interval $I\subset \Aset$ and some positive integer $\ell$ such that $x+\ell\alpha \in \Aset$ for every $x\in I$ and, if $\ell>1$, that
\begin{equation}
\label{eq:E_not_eigenvalue}
\mathrm{dist}\left(E, \sigma\left(H_{[1,\ell-1]}(x)\right)\right)\geq \delta, \qquad \forall x\in I.
\end{equation}
 For $x\in \mathbb{T}$ denote then $\NA(x)=\{i \in \mathbb{Z}: \; x+i\alpha \in \Aset\}$.

		For fixed $C$ and $\alpha$ there then exists a constant $\eta_{\delta, I, \ell}>0$ depending on $\delta$, $I$, and $\ell$ such that 
\begin{equation}
\label{eq:mass_lowerbound1}
\sum_{i\in\NA(x)}|\psi(i)|^2 \ \ge\ \frac{\eta_{\delta,I, \ell}}{1+\eta_{\delta,I, \ell}}
\end{equation}
uniformly for all sufficiently large $n$.
\end{lemma}
\begin{proof}
There exists a uniform constant $L$ such that for every $x\in \mathbb{T}$ the orbit $(x, x+\alpha, x+2\alpha,..., x+L\alpha)$ contains a point of $I$ and $(x, x-\alpha, x-2\alpha,..., x-L\alpha)$ contains a point of $I+\ell \alpha$. Assuming $n$ is large enough there is then for every $i \in [0,n-1]$ an interval $[m_1,m_2]\subset [0,n-1]$ with $\mathrm{dist}(i, [m_1,m_2])\leq L$ and $x+m_1\alpha\in I$.

Write $a_i:=E-f(x+i\alpha)$. Assume $i \in [0,n-L-1]$ then the recurrence for $i+1$ gives
\[
\psi(i)=a_{i+1}\psi(i+1)-\psi(i+2).
\]
As in the proof of Proposition~\ref{proposition:hf_bound_warmup} one obtains $$\max(|\psi(i+1)|^2, |\psi(i+2)|^2) \geq  \frac{1}{(|a_{i+1}|+1)^2} |\psi(i)|^2\geq \tilde{\eta} |\psi(i)|^2$$
with $\tilde\eta:=\frac{1}{(3+2 C)^2}$. By iterating this estimate at most $L$ times we can conclude that for any $i \in [0,n-L-2] $ and $j\in 1,...,L$ one gets
\[
\max\left(|\psi(i+j)|^2,|\psi(i+j+1)|^2\right) \ \geq \tilde\eta^{L}\,|\psi(i)|^2
\]
i.e. out of any two consecutive sites within $L+1$ steps of $i$ at least one must be larger than the right-hand side.

To also treat the upper end of the interval we can instead use the backward recurrence to conclude that for any $i\in [L+1,n-1]$ and $j\in 1,...,L$ one has
\[
\max(|\psi(i-j)|^2,|\psi(i-j-1)|^2) \geq \tilde\eta^{L}\,|\psi(i)|^2.
\]

We conclude that for every $i \notin \NA(x)$ there exists an interval $[m_1, m_2]$, $m_2=m_1+\ell$ with $x+m_1\alpha \in I$ and a site $m \in [m_1,m_2]$ with
$$|\psi(m)|^2 \geq \tilde\eta^{L}\,|\psi(i)|^2.$$
If $m_1 < m < m_2$ then by \eqref{eq:E_not_eigenvalue} and Lemma~\ref{lemma:poisson} we have
$$\max\{|\psi(m_1)|^2,|\psi(m_2)|^2\} \;\ge\; \frac{\delta^2}{4}\,|\psi(m)|^2 \geq \frac{\delta^2}{4} \tilde{\eta}^{L}\,|\psi(i)|^2$$
and thus including the edge case $m=m_1$ or $m=m_2$ we have
$$\max\{|\psi(m_1)|^2,|\psi(m_2)|^2\} \;\ge\; \min\left(1,\frac{\delta^2}{4}\right) \tilde{\eta}^{L}\,|\psi(i)|^2.$$
Note that $m_1,m_2\in N_{\Aset}(x)$ by assumption. Summing this inequality over $i\notin \NA(x)$ and accounting for the fact that each index $m$ can be produced by at most $2L$ different $i$ we obtain
\[
2L
\sum_{m\in\NA(x)}|\psi(m)|^2 \geq \min\left(1,\frac{\delta^2}{4}\right)  \tilde{\eta}^{L}\sum_{i\notin \NA(x)}|\psi(i)|^2
\]
which can be rearranged to \eqref{eq:mass_lowerbound1} using $\ell^2$-normalization of $\psi$.
\end{proof}

The proposition is formulated in this precise way since we eventually need to prove a lower bound which is uniform in $E$, but depending on $E$ we may need to use a different sub-interval $I\subset \Aset$ such that \eqref{eq:E_not_eigenvalue} can hold.

For a set $\Aset\subset \mathbb{T}$ denote by $\ell_\Aset(x)$ the return time for $x\in \Aset$, i.e. the smallest positive integer such that $x+\ell_\Aset(x)\alpha\in \Aset$.

\begin{proposition}
\label{prop:mass_lower_bound}
Consider the finite-volume Hamiltonian $H_n(x)$ on the integer interval $[0,n-1]$ with a non-negative bounded piecewise-continuous sampling function $f$.

Let $\Aset\subset \mathbb{T}$ be a non-empty open interval. If all $x\in \Aset$ satisfy $\ell_\Aset(x)>1$ then let $\tilde{\Aset}$ be the set of continuity points of the map $$x \in \Aset \mapsto H_{[1,\ell_{\Aset}(x)-1]}(x)$$ and assume that there does not exist any $E\in \RM$ which is simultaneously an eigenvalue of all of the matrices
$$\left(H_{[1,\ell_\Aset(x)-1]}(x)\right)_{x\in \tilde{\Aset}}.$$ Then there exists a constant $\eta>0$ such that for any normalized eigenvalue-eigenvector pair $(E,\psi)$ of $H_n(x)$ one has
\begin{equation}
\label{eq:mass_lowerbound}
\sum_{i\in\NA(x)}|\psi(i)|^2 \ \ge\ \frac{\eta}{1+\eta}
\end{equation}
uniformly in $x\in \mathbb{T}$ and for all sufficiently large $n$.
\end{proposition}
\begin{proof}
If there exists some $x\in \Aset$ with $\ell_\Aset(x)=1$ then this means $x+\alpha \in \Aset$ and since $\Aset$ is open there exists an interval $I$ with $I\subset (\Aset-\alpha)\cap \Aset$. We can then apply Lemma~\ref{lemma:poisson_mass_transfer} with that interval $I$ for every $E$.

Conversely, if all $x\in \Aset$ have $\ell_\Aset(x)>1$ then $x \mapsto H_{[1,\ell_\Aset(x)-1]}(x)$ is piecewise and norm-continuous, therefore the absence of a common eigenvalue implies that there exists for every $E\in \mathbb{R}$ some $\delta>0$ and some interval $I_E \subset \Aset$ with $\mathrm{dist}(E, \sigma(H_{[1,\ell_\Aset(x)-1]}(x)))>2\delta$ for every $x\in I_E$. Then one also has $\mathrm{dist}(E', \sigma(H_{[1,\ell_\Aset(x)-1]}(x)))>\delta$ for every $E'\in (E-\delta, E+\delta)$. As $\sigma(H_n(x))\subset [-2,2+\lVert f\rVert]$ we can by compactness pick finitely many intervals $I_1,...,I_N\subset \Aset$ such that for every $E \in [-2,2+\lVert f\rVert]$ there exists some interval $I_j \subset \Aset$ such that $\mathrm{dist}(E, \sigma(H_{[1,\ell_\Aset(x)-1]}(x)))>\delta$ for every $x\in I_j$. By making each $I_j$ smaller if necessary we may assume $\ell_\Aset\rvert_{I_j}$ is a constant that we denote $\ell_j$. From Lemma~\ref{lemma:poisson_mass_transfer} and \eqref{eq:mass_lowerbound1} we therefore conclude
$$\sum_{i\in\NA(x)}|\psi(i)|^2 \ \ge\  \frac{\eta}{1+\eta}$$
with $\eta = \min_{j=1,...,N}(\eta_{\delta,I_j, \ell_j})$.
\end{proof}

We can now finish the proof of Theorem~\ref{thm:bounded_implies_A1A2}:

\begin{proof}[Proof of Theorem~\ref{thm:bounded_implies_A1A2}]
Choose an interval $I\subset \Aset \cap (\Aset - \alpha)$, then the conditions of Lemma~\ref{lemma:poisson_mass_transfer} hold for this given $I\subset \Aset$ uniformly in $E$ with $\ell=1$, i.e. \eqref{eq:E_not_eigenvalue} is vacuous. The interval $I$ and the constant $\eta$  are uniform for all absolutely continuous sampling functions with $\lVert f\rVert_\infty \leq C$. For almost every $x$ we then have
$$E_k'(x)=\sum_{i\in \NA(x)} |\psi_k(i)|^2\, f'(x+i\alpha) +  \sum_{i\notin \NA(x)} |\psi_k(i)|^2\, f'(x+i\alpha) \geq \frac{\eta}{1+\eta}\gamma_+ + \frac{1}{1+\eta}\min(0,\gamma_-).$$
If the sampling function is piecewise absolutely continuous then the eigenvalue curves are also piecewise absolutely continuous. In between endpoints $\beta_j,\beta_{j+1}$ of a monotonicity interval all discontinuities of $E_k(x)$ come from jump discontinuities of $f$ which are assumed to be positive. Therefore we have
		$$E_k(y)- E_k(x) \geq \int_{x}^y E_k'(z) dz \geq \left(\frac{\eta}{1+\eta}\gamma_+ + \frac{1}{1+\eta}\min(0,\gamma_-)\right) (y-x)$$ 
for all $x<y$ within the same monotonicity interval, which gives a strictly positive anti-Lipschitz constant if
$|\gamma_-| < \eta \gamma_+$. 

If the sampling function is monotone but not necessarily piecewise absolutely continuous then the eigenvalue curves $x\mapsto E_k(x)$ are also monotone and hence have distributional derivatives which are positive measures. Hence, one still has a lower bound given by the absolutely continuous part
$$E_k(x)-E_k(y) \geq \int_{y}^x E'_k(z) dz \geq \frac{\eta}{1+\eta}\gamma_+\cdot(x-y).$$
\end{proof}
One could in this case also derive an explicit bound for $\eta$ and thus the anti-Lipschitz constant in terms of the length of the interval $I\subset \Aset \cap (\Aset-\alpha)$ and the continued fraction expansion of $\alpha$, since that determines the length $L$ in the proof of Proposition~\ref{prop:mass_lower_bound}.

\subsection{Locally constant reduction}
\label{sec:lcreduction}
For the proof of Theorem~\ref{thm:monotonederivative_implies_A1A2} we need to assert that for a (piecewise) continuous monotone sampling function one can always find an interval such that the finite-volume restrictions \eqref{eq:E_not_eigenvalue} do not have a common eigenvalue. While one may be quick to think that this should be \emph{generic} in some sense this is actually far from obvious. In the most difficult to handle cases the sampling function has vanishing derivative on most of $(0,1)$ and is therefore precisely constant on some intervals:

\begin{example}
For some interval $[a,b]\subset [0,1]$ consider the piecewise linear sampling function
$$f(x)=\begin{cases}
0 & \text{ if }x<a\\
v & \text{ if }x>b\\
v \frac{x-a}{b-a} & \text{ if }x\in [a,b]
\end{cases}$$
for some potential value $v>0$.
Let us fix $\Aset=I=[a,b]$ then by definition of the return time $\ell_\Aset(x)$ one has for any $x\in I$ that in the trajectory $(x, x+\alpha, ..., x+\ell_\Aset(x)\alpha)$ precisely $x$ and $x+\ell_\Aset(x)\alpha$ are in $I$ and the remaining elements are not. Therefore the potential of any finite-size restriction $H_{[1,\ell_\Aset(x)-1]}(x)$ for $x\in I$ takes exactly two different values $0$ and $v$, i.e. it is a word determined by the quasi-periodic dynamics. 

For such a locally constant potential only finitely many of the Dirichlet restrictions $H_{[1,\ell_\Aset(x)-1]}(x)$ are actually distinct and we can label them by words in a finite alphabet. Hence it seems quite possible in principle that they could have an accidental eigenvalue coincidence, or maybe even a systematic one if the words satisfy nontrivial combinatorial relations. However, it will turn out the opposite is true! For generic values of $a,b$ the intricate structure of the words actually makes it completely impossible to have a common eigenvalue, one cannot even get an accidental coincidence if one fine-tunes the values of the potentials.

\end{example}

As we will prove in this section, one can reduce the question of common Dirichlet eigenvalues for an arbitrary continuous monotone sampling function to that of a locally constant sampling function which takes only finitely many different values, hence this presents in some concrete sense the worst case scenario. Let therefore $\mathbb{T}=\bigcup_{k=0}^{n} I_k$ be a decomposition into finitely many half-open intervals which can hence be ordered as arcs $I_k=[a_k,a_{k+1})$ with $a_{n+1}=a_0$. We single out one \emph{good} interval $G=I_0$ and consider return words with respect to it. For any $x\in G$ we can define the return time $\ell_G(x)$ as the shortest positive time such that $x+\ell_G(x)\alpha\in G$. Fix some alphabet $\mathcal{A}$ and associate to every interval $I_k$ some distinct letter $l_k\in \mathcal{A}$, then set $\mathcal{L}_G(x)=l_k$ whenever $x\in I_k$. The return word $w_G(x)$ for $x\in G$ is given by
$$w_G(x)= \mathcal{L}_G(x+\alpha)\mathcal{L}_G(x+2\alpha)...\mathcal{L}_G(x+(\ell_G(x)-1)\alpha)$$
as the consecutive letters of the intervals visited on the return orbit $x+\alpha,...,x+(\ell_G(x)-1)(\alpha)$. It is easy to see that $w_G(x)$ is locally constant in $x$ and in fact there are always only finitely many different return words.

In our setting collections of intervals correspond to sampling functions $f$ which are locally constant outside some interval $G$ (corresponding to some interval where the derivative $f'$ has a positive lower bound). The corresponding Dirichlet Hamiltonians are then abstractly given as follows:
\begin{definition}
\label{def:word_hamiltonian}
Let $W$ be a finite set of words in an alphabet $\mathcal{A}$. We say that $W$ has a common Dirichlet eigenvalue if there is an association of numbers $v: \mathcal{A}\to \mathbb{R}$ such that the Dirichlet Hamiltonians 
$$(H_w\psi)(n) = (\Delta\psi)(n) + v(w_n)\psi(n)$$
on $\ell^2[1, \abs{w}]$ have a common eigenvalue $E$ for all words $w\in W$.
\end{definition}
We rule out the special case of some word $w$ being the empty word as the collection having no common Dirichlet eigenvalue. Our main result is then as follows:
\begin{theorem}
\label{theorem:nodirichlet}
Assume that no two endpoints $a_0, ..., a_n$ of the intervals $I_0,..., I_n$ are on the same $\alpha$-orbit and that the interval $G=[a_n,a_0)$ has length less than $1-\|\alpha\|$. Then $G=I_0$ has exactly $n+2$ distinct return words $W=\{w_G(x): x\in G\}$ and each word occurs on some sub-interval of $G$ which has positive length. The words $W$ do not have a common Dirichlet eigenvalue. 
\end{theorem}
The condition is sufficient but not necessary. One can in fact construct counterexamples if one allows the endpoints to be on the same $\alpha$-orbit, hence the conditions can possibly be weakened, but not removed entirely. A more detailed investigation can be found in an upcoming work \cite{ReturnWords}. Note also that we do not impose monotonicity here on the functions which associate potentials to letters. It is even less clear under what conditions one can have counterexamples which respect monotonicity.

Before proving this result in the next section let us first explain how this finishes the proof of Theorem~\ref{thm:monotonederivative_implies_A1A2}. First recall an elementary fact about return times of intervals.
\begin{lemma}
\label{lemma:returntimes}
Let $\alpha$ be irrational, then the return times $\ell_G$ of an interval $G=[a,b)$ with $|G|<1-\|\alpha\|$ to itself partition $G$ into two or three different intervals with distinct relatively prime return times.
\end{lemma}
\begin{proof}
The return times can be written in terms of the continued fraction expansion \cite{Slater1950}. There are two or three return times with the shortest given by the first denominator $r_1 = q_n$ of $\alpha$ such that $\|q_n\alpha\|<|b-a|$. If there are two return times the larger is of the form $r_2 = q_{n+1} - k q_n$ and if there are three then the larger two are of the form $r_2 = q_{n+1} - k q_n$ and $r_3 = q_{n+1} - (k+1) q_n$, hence relatively prime. That the return times partition $G$ into intervals will be proven in more generality later in Lemma~\ref{lemma:connected_component_words} so we omit the proof.
\end{proof}
We also require a preparatory Lemma about monotone families of potentials:
\begin{lemma}
\label{lemma:common_eigenvector}
Let $I=[a,b]$ be an interval and $x\in I \mapsto H(x)= \Delta_\ell + V(x) \in M_\ell(\mathbb{C})$ a family of self-adjoint matrices with $\Delta_\ell$ the discrete Dirichlet-Laplacian on $[1,\ell]$ and $x\mapsto V(x)$ a continuous monotone family of diagonal matrices with $V(x) \leq V(y)$ for $x<y$. If there exists some $E\in \mathbb{R}$ which is an eigenvalue for all $H(x)$ then there exists a common eigenvector $\psi$ with $H(x)\psi=E\psi$ for all $x$.

Moreover, one has $\psi(j)=0$ for every $1\leq j \leq \ell$ such that the diagonal entry $x\in I\mapsto V_j(x)$ is not constant.
\end{lemma}
\begin{proof}
Because all eigenvalues are simple the eigenvalue curves are continuous and  do not intersect. Therefore one has $E=\lambda_k(H(x))$ for some $k$, i.e. it is the $k$-th eigenvalue independent of $x$. Pick any $x$ in $I$ and consider the path $$t\in [0,1]\mapsto h_t= (1-t)H(a) + t H(x)= \Delta_\ell + V(a) + t (V(x)-V(a)).$$ The path $h_t$ also is monotone and has no eigenvalue intersections, moreover $\lambda_k(h_0)= E = \lambda_k(h_1)$. This is only possible if $\lambda_k(h_t)=E$ for all $t$ since $t\to\lambda_k(h_t)$ is monotone. Let now $\psi$ be a normalized eigenvector with $H(a)\psi=E\psi$. By Hellmann-Feynman applied at $t=0$ we conclude $$\langle \psi, (V(x)-V(a))\psi\rangle = \frac{d}{dt} \lambda_k(h_t)\big\rvert_{t=0} = 0$$ for every $x$. In particular, since
$$\langle \psi, (V(x)-V(a))\psi\rangle= \sum_{j=1}^\ell \lvert V_j(x)-V_j(a)\rvert \; \lvert\psi(j)\rvert^2 = 0 $$ we have $\psi(j)=0$ for every $1\leq j \leq \ell$ such that $V_j(x)\neq V_j(a)$, which implies $V(x)\psi=V(a)\psi$ and hence $\psi$ is actually an eigenvector for all $H(x)$. 
\end{proof}
In particular this implies that a monotone family can only have a common eigenvalue if some of the diagonal entries are actually constant. We can use this principle to reduce the general case of a continuous monotone sampling function to the locally constant case:

\begin{proposition}
\label{prop:monotone_no_eigenvalue}
Let $f: \mathbb{T}\to \mathbb{R}$ be a monotone piecewise continuous sampling function with an interval $I$ such that $\essinf_{x\in I}f'(x)>0$ and $\ell_I(x)>1$ for every $x\in I$. Moreover, assume that no two discontinuity points of $f$ are on the same $\alpha$-orbit. Then there exists some interval $G\subset I$ such that the family
$$(H_{[1,\ell_G(x)-1]}(x))_{x\in \tilde{G}}$$ does not have a common eigenvalue, where $\tilde{G}\subset G$ is the set of continuity points of the map $x\mapsto H_{[1,\ell_G(x)-1]}(x)$.
\end{proposition}
\begin{proof}
Pick any interval $G=[a,b) \subset I$ such that $a,b$ are not on the same $\alpha$-orbit nor on the same $\alpha$-orbit as $0=1$ nor any other of the discontinuity points of $f$.

Assume for a contradiction that $(H_{[1,\ell_G(x)-1]}(x))_{x\in \tilde{G}}$ has a common eigenvalue $E$, then there exists a continuous family of normalized eigenvectors $(\psi_x)_{x\in \tilde{G}}$  with a common eigenvalue $E$. Here we excluded a finite set of discontinuity points $G_0 =G\setminus \tilde{G}$ that we can write as $G_0=\{z_1,...,z_N\}$ with some ordered points $z_1<z_2<...<z_N$. Note that this set is finite, since by Lemma~\ref{lemma:returntimes} $\ell_G(x)$ takes at most three different values $r_1,r_2,r_3$ on at most three distinct subintervals, hence the number of discontinuities is smaller than $N\leq (r_1+r_2+r_3)(1+N_c)$ where $N_c$ is the number of discontinuity points of $f$ on $(0,1)$. On each interval $(z_i, z_{i+1})$ the map $x\in (z_i,z_{i+1})\mapsto H_{[1,\ell_G(x)-1]}(x)$ is continuous and monotone, by Lemma~\ref{lemma:common_eigenvector} we may therefore assume without loss of generality that $\psi_x$ is constant for $x$ inside any interval $(z_i, z_{i+1})$. For any $x\in \tilde{G}$ define the value set $$\mathcal{V}(x)=\{f(x+j\alpha): 1\leq j \leq \ell_G(x)-1\;\;\wedge \;\; \psi_x(j)\neq 0\}$$
as those values of $f$ which are actually sampled by the corresponding eigenvector. The map $x\in (z_i,z_{i+1}) \mapsto \mathcal{V}(x)$ is constant for every $i$, as the final statement of Lemma~\ref{lemma:common_eigenvector} implies that the common eigenvector for the continuity interval vanishes on all sites $j$ such that $x\in (z_i,z_{i+1})\mapsto f(x+j\alpha)$ is not constant. In conclusion, the total value set $\mathcal{V}:= \bigcup_{x\in \tilde{G}}\mathcal{V}(x)$ is a finite set.

Importantly, one has $\psi_x(j)=0$ whenever $f(x+j\alpha)\notin \mathcal{V}$, or equivalently, $$x+j\alpha\notin f^{-1}(\mathcal{V}) =: \Sigma.$$ Hence, whenever $\tilde{f}$ is a sampling function with $f\rvert_{\Sigma}= \tilde{f}\rvert_\Sigma$ then $\{\psi_x\}_{x\in \tilde{G}}$ is also a family of eigenvectors for the associated Hamiltonians $(\tilde{H}_{[1,\ell_G(x)-1]}(x))_{x\in \tilde{G}}$ still with the same common eigenvalue $E$. We will use this to derive a contradiction to Theorem~\ref{theorem:nodirichlet}.

As $f$ is monotone we can write the pre-image of $\mathcal{V}$ as a finite union of disjoint intervals, in particular its closure is
$$\overline{\Sigma} =\bigcup_{i=0}^n [x_i, y_i]$$
with endpoints $0\leq x_0 \leq y_0 \leq x_1 \leq y_1 \leq ... \leq x_n \leq y_n\leq 1$ and then $f$ is constant on the interior of each interval $(x_i,y_i)$ with some value $f\rvert_{(x_i,y_i)}=c_i$. We can pick successively elements $y_{i-1} \leq a_i \leq x_{i}$ such that in the collection $\{a_i\}_{i=1,...,n}$ no two elements are on the same $\alpha$-orbit nor on the same $\alpha$-orbit as $a,b$ or $0=1$. This is always possible since we are picking $a_i$ either from an open interval (if $y_{i-1}<x_i$) or it is fixed to be one of the discontinuity points of $f$ (if $y_{i-1}=x_i$) and those are on distinct $\alpha$-orbits by assumption. In addition we set $a_0=0$ and $a_{n+1}=1$. Setting
$$\tilde{f}(x) = \begin{cases}
f(x) & \text{if }x\in \Sigma\\
c_i & \text{if }x\in [a_i, a_{i+1}) \setminus \Sigma
\end{cases}$$
we obtain a function with $f\rvert_{\Sigma}=\tilde{f}\rvert_{\Sigma}$ and such that the finite-size restrictions $(\tilde{H}_{[1,\ell_G(x)-1]}(x))_{x\in \tilde{G}}$ only sample $\tilde{f}$ on the intervals $\tilde{f}^{-1}(c_i)$ where it is locally constant. Those intervals where $\tilde{f}$ is constant are not necessarily half-open, but the distinction becomes irrelevant if we exclude a countable set: When we set $\hat{G}\subset \tilde{G}$ with $\tilde{G}\setminus \hat{G}$ given by all points on the same $\alpha$-orbit as the endpoints $a_0,...,a_n$ then one obtains
$$\{\tilde{H}_{[1,\ell_G(x)-1]}(x): x\in \hat{G}\}= \{H_w: w\in W\} $$ with Hamiltonians as in Definition~\ref{def:word_hamiltonian} and $W$ a collection of return words of $G$ with respect to a finite family of labeled half-open intervals $(I_j)_{j\in J}$ whose endpoints lie in $\{a_k\}\cup \{a,b\}$. By construction none of these endpoints are on a common $\alpha$-orbit and therefore all return words of $(I_j)_{j\in J}$ occur on subintervals of $G$ of positive length. Therefore $W$ is in fact the collection of \emph{all} possible return words with respect to $(I_j)_{j\in J}$ as the countably many exclusions $G\setminus \hat{G}$ cannot cause us to miss a return word. We conclude that the family having $E$ as a common eigenvalue is a contradiction to  Theorem~\ref{theorem:nodirichlet}.
\end{proof}

\begin{proof}[Proof of Theorem~\ref{thm:monotonederivative_implies_A1A2}]

This is essentially the same as the proof of Theorem~\ref{thm:bounded_implies_A1A2}.
		Let $\Aset$ be the interval $G$ constructed in Proposition~\ref{prop:monotone_no_eigenvalue}. Shrinking $I$ if necessary we may assume that $\ell_I(x)>1$ for every $x\in I$. We may then apply Proposition~\ref{prop:mass_lower_bound} as the absence-of-common-eigenvalue condition holds by Proposition~\ref{prop:monotone_no_eigenvalue}. Hence there exists
a constant $\eta>0$ such that
$$E_k'(x)=\sum_{i\in \NA(x)} |\psi_k(i)|^2\, f'(x+i\alpha) +  \sum_{i\notin \NA(x)} |\psi_k(i)|^2\, f'(x+i\alpha) \geq \frac{\eta}{1+\eta}\gamma_+$$
with $\gamma_+=\essinf f'\rvert_G$.
\end{proof}

\begin{proof}[Proof of Theorem~\ref{thm:monotonederivative_plus_perturbation_implies_A1A2}]
Let $\Aset$ be the interval $G$ constructed in Proposition~\ref{prop:monotone_no_eigenvalue} for the monotone increasing part $f$. As the Dirichlet restrictions $$(H_{[1,\ell_G(x)-1]}(x))_{x\in \tilde{G}}$$ for the sampling function $f$ do not have a common eigenvalue and are piecewise continuous there exists a constant $c>0$ such that the Dirichlet restrictions
$$(\tilde{H}_{[1,\ell_G(x)-1]}(x))_{x\in \tilde{G}}$$
based on the sampling function $\tilde{f}=f+g$ also do not have a common eigenvalue as long as $\|g\|_\infty < c$ is small enough.

By Proposition~\ref{prop:monotone_no_eigenvalue} the conditions of Lemma~\ref{lemma:poisson_mass_transfer} on absence of common eigenvalues hold and therefore there exists some constant $\eta>0$ such that the eigenvalue curves with respect to $\tilde{f}$ satisfy
$$\tilde{E}_k'(x)=\sum_{i\in \NA(x)} |\psi_k(i)|^2\, \tilde{f}'(x+i\alpha) +  \sum_{i\notin \NA(x)} |\psi_k(i)|^2\, \tilde{f}'(x+i\alpha) \geq \frac{\eta}{1+\eta}\tilde\gamma_+ +  \frac{1}{1+\eta}\tilde\gamma_-$$
with $\tilde{\gamma}_+=\essinf \tilde{f}'\rvert_G$ and $\tilde{\gamma}_-=\essinf \tilde{f}'$ and some constant $\eta>0$ which can be chosen uniformly for all functions $g$ with $\|g\|_\infty < c$. Since $\gamma_+:=\essinf f'\rvert_G>0$ this is bounded from below by a positive constant if $\|g'\|_\infty$ is small enough depending on $\eta$ and $\gamma_+$.
\end{proof}

\subsection{Absence of common eigenvalues}
\label{sec:absence}
In this section we prove Theorem~\ref{theorem:nodirichlet} and therefore consider $n+1$ finite intervals with endpoints $a_0,...,a_n$. For convenience we will for this section choose to parametrize $\mathbb{T}$ in such a way that the interval $G=I_0=[a_n, a_0)$ always includes the wrap-around $0=1$ since then all points outside $G$ are totally ordered and the intervals $I_k=[a_k, a_{k+1})$ are in sequence with $a_0,...,a_n \in (0,1)$ and $a_k<a_{k+1}$. 
The proof will eventually be by induction, hence it is instructive to consider the simplest special case first:

\begin{proposition}
\label{proposition:no_common_ev_dirichlet_laplacian}
For two intervals $G=I_0= [0,a)$ and $I_1=[a,0)$ the return words $W=\{w_G(x)\;:\;x\in G\}$ do not have a common Dirichlet eigenvalue.
\end{proposition}
\begin{proof}
By Lemma~\ref{lemma:returntimes} the set $W$ has either two elements $W=\{l^{r_1-1}, l^{r_2-1}\}$ or three elements $W=\{l^{r_1-1}, l^{r_2-1}, l^{r_3-1}\}$ where $l$ is the letter associated to $I_1$ and $r_i$ are the return times of $G$ of which there are either two or three distinct ones. Since there is only one letter the Hamiltonians are up to a constant shift the discrete Dirichlet Laplacians on intervals of length $|w|=r_i-1$. Since the return times are relatively prime the eigenvalues of the Dirichlet Laplacians given by solutions of the equations
$$E_k=2\cos\left(\frac{k\pi}{r_i}\right), \qquad k=1,...,r_i-1$$
never have an eigenvalue in common.
\end{proof}

For more complex words the following stripping principle will be useful:
\begin{lemma}
\label{lemma:stripping}
Let $W_1$ and $W_2$ be two collections of words of the same alphabet such that for any $\tilde{w}\in W_2$ one either has $\tilde{w}\in W_1$ or there exist two words $v,w\in W_1$ such that $w=\tilde{w}lv$ or $w=vl\tilde{w}$ with some letter $l$.

If $W_2$ has no common Dirichlet eigenvalue then $W_1$ also has no common Dirichlet eigenvalue.
\end{lemma}
\begin{proof}
For the contrapositive let $W_1$ have a common Dirichlet eigenvalue. Then by assumption no word in $W_1$ is empty. Then we can assume an association of values to letters such that for all $w\in W_1$ the Hamiltonians $H_w$ have some common eigenvalue $E$. If $\tilde{w} \in W_1\cap W_2$ is a shared word then any Hamiltonian $H_{\tilde{w}}$ trivially also has eigenvalue $E$. For a new word $\tilde{w}\in W_2\setminus W_1$ we assume $w=\tilde{w}lv$ with the other order being similar. By a simple transfer matrix argument any non-zero Dirichlet eigenvector $\psi$ of $H_w$ with energy $E$ is equal to a Dirichlet eigenvector of $H_v$ when restricted to the final $|v|$ sites and vanishes on the site adjacent to $v$, i.e. the one corresponding to $l$ here. The restriction of $\psi$ to the first $|\tilde{w}|$ letters therefore also satisfies Dirichlet boundary conditions, hence $E$ is also an eigenvalue of $H_{\tilde{w}}$ if $\tilde{w}$ is non-empty. And indeed $\tilde{w}$ cannot be empty, since then the solution $\psi$ would be zero on the first two sites of $w=lv$ due to Dirichlet boundary conditions, hence would have to vanish identically by a transfer matrix argument.
\end{proof}

We will need to understand in detail how many different return words occur and how they split $G$ into subintervals with different words. For this the half-open convention is very useful, since it implies that translates of $G$ cover the torus before returning to $G$:

\begin{lemma}
\label{lemma:partition}
Let $G=[a,b)\subset \mathbb T$ be an interval with $a\neq 0\neq b$, and let $T(x)=x+\alpha \pmod 1$ 
with $\alpha$ irrational. Define
\[
G_1 := T(G), \qquad 
G_{i+1} := T(G_i\setminus G) \quad (i\ge 1).
\]
Let $\ell_G(x)=\min\{n\ge 1 : T^n x\in G\}$ be the first return time to $G$ and
\[
r := \max\{\ell_G(x) : x\in G\}.
\]
Then the sets $G_i$ for $1\le i\le r$ are pairwise disjoint half-open intervals,
$G_i\neq \emptyset$ for $1\le i\le r$, $G_{r+1}=\emptyset$, and
\[
\mathbb T = \bigsqcup_{i=1}^r G_i.
\]

In conclusion, there is for every $x\in \mathbb{T}$ some unique $x_0\in G$ whose orbit meets $x$ before returning to $G$ and the unique index $i$ such that $x\in G_i$ is the time it takes $x_0$ to reach $x$, i.e. $T^i(x_0)=x$.

\end{lemma}

\begin{proof}
It is clear that all $G_i$ are half-open intervals, since no $G_i$ can be equal to $G$ for irrational $\alpha$, so the operation $G_i\setminus G$ always results in a half-open interval or an empty set.

For $i\ge 1$ define
\[
I_i := T^{-i}(G_i)\subset G,
\]
so that $G_i=T^i(I_i)$. We prove by induction that
\begin{equation}
\label{eq:lemma_return_intervals_tech}
I_i
=
\{x\in G : T^k x\notin G \text{ for } k=1,\dots,i-1\}.
\end{equation}
For $i=1$ this is clear since we have $G_1=T(G)$ and hence $I_1=T^{-1}(G_1)=G$. Assume \eqref{eq:lemma_return_intervals_tech} holds for some $i\ge 1$. From the recursion
\[
G_{i+1}=T(G_i\setminus G)=T(T^i(I_i)\setminus G)
\]
applying $T^{-(i+1)}$ gives
\[
I_{i+1}
=
T^{-(i+1)}(G_{i+1})
=
T^{-i}(T^i(I_i)\setminus G)
=
I_i \setminus T^{-i}(G).
\]
Thus $x\in I_{i+1}$ if and only if $x\in I_i$ and $T^i x\notin G$. By the induction hypothesis, this is equivalent to $T^k x\notin G$ for $k=1,\dots,i$.

Using \eqref{eq:lemma_return_intervals_tech} we can now prove that the intervals $G_i$ are disjoint.
Suppose $y\in G_i\cap G_j$ with $i<j$.
Let $x:=T^{-j}y\in I_j\subset G$. Then
\[
T^{j-i}x = T^{-i}y \in T^{-i}(G_i)=I_i\subset G.
\]
Since $1\le j-i\le j-1$, this contradicts \eqref{eq:lemma_return_intervals_tech} for $I_j$.
Hence the sets $G_i$ are pairwise disjoint.

For the fact about the return time, note that $r$ is the smallest positive number such that $I_{r+1}=\emptyset$ and therefore $I_i\neq \emptyset$ for all $i=1,...,r$. Equivalently, one has $G_{r+1}=\emptyset$ and $G_i\neq \emptyset$ for all $i=1,...,r$.

Finally, we prove the sets $G_i$ cover the torus. For any $x\in \mathbb{T}$ there exists some $n$ such that $T^{-n}x\in G$ since $G$ contains an open interval and $\alpha$ is irrational. For the minimal of such $n$ one then has $x\in G_n$.
\end{proof}

We assume throughout that $a_0,\ldots,a_n$ all have disjoint $\alpha$-orbits. 

\begin{lemma}
\label{lemma:connected_component_words}
For $0\leq k\leq n$ define $$x_k=a_k-j_k\alpha$$
where $j_k\geq 1$ is the smallest positive integer with the property $a_k-j_k\alpha\in [a_n
,a_0)$.

Assume $|G|<1-\|\alpha\|$, then the points $x_0,\ldots,x_n$ divide the interval $G$ into $n+2$ connected components. On each connected component, $w_G$ is constant. Furthermore, for $x,y\in G$ which are in different connected components, we have $w_G(x)\neq w_G(y)$.
\end{lemma}
\begin{proof}
For each $0\leq k\leq n$, the point $x_k$ is the unique point in $G$ whose orbit hits $a_k$ before returning to $G$. Indeed, by the previous lemma the sets $G_i$ form a partition of $\mathbb T$, so every point of $\mathbb T$ has a unique pre-return representative in $G$. Applying this to $a_k$ gives the uniqueness of $x_k$. The $x_k$ are distinct because of the disjoint $\alpha$-orbits. If $|G|>\alpha$ then some return words can be empty, however, the condition $|G|<1-\|\alpha\|$ ensures that there is at most one empty word (either the leftmost or rightmost), hence uniqueness of the return words will still hold. 

We claim that $w_G$ is constant on each connected component of
\[
G\setminus \{x_0,\dots,x_n\}.
\]
Suppose not. Then there exist two points $x,y$ in the same connected component such that
\[
w_G(x)\neq w_G(y).
\]
Let $j$ be the first position at which the two return words differ. Then one of two things happens: Either one of the words has length less than $j$ (so the $j$-th letter disappears) or both words have length at least $j$ (in which case their $j$-th letters are different). As one moves continuously from $x$ to $y$ inside the same connected component, at the moment when this change first occurs, the point $T^j(z)$ must lie on one of the boundary points $a_k$ for some intermediate point $z$. Indeed, the only way a letter can disappear is that the orbit returns to $G$ exactly at that step, and the only way a letter can change is that the orbit crosses from one interval $I_{k-1}$ to the neighboring interval $I_k$, which again can only happen by passing through the common boundary point $a_k$.

But by the uniqueness established above, the only point in $G$ whose orbit hits $a_k$ before returning to $G$ is $x_k$. Hence $z=x_k$, contradicting the fact that $x$ and $y$ lie in the same connected component of
\[
G\setminus \{x_0,\dots,x_n\}.
\]
This proves that $w_G$ is constant on each connected component.

It remains to show that different connected components correspond to different return words. 
First note that the return length can change only at the points $x_0$ and $x_n$, since these are exactly the pre-return representatives of the two endpoints of $G$. Hence the return length is constant on each connected component of $G\setminus\{x_0,x_n\}$. In particular, if two connected components of
$
G\setminus\{x_0,\dots,x_n\}
$
lie in different connected components of
$
G\setminus\{x_0,x_n\},
$
then the corresponding return words have different lengths and therefore are different.

It remains to compare two connected components lying in the same connected component of
$
G\setminus\{x_0,x_n\}.
$
On such an interval, all return words have the same length. For each fixed position $m$ in the word, the $m$-th letter is determined by which interval $I_j$ contains the point $T^m x=x+m\alpha$.
Since $x\mapsto x+m\alpha$ is a translation, the preimage of each interval $I_j$ is again an interval. Therefore, as $x$ moves from left to right, the $m$-th letter can change only by passing from one interval $I_j$ to the next, and in particular it cannot later return to a previous value.

Now let $J_1,J_2$ be two distinct connected components of $G\setminus\{x_0,\dots,x_n\}$ lying in the same connected component of $G\setminus\{x_0,x_n\}$.
Then moving from $J_1$ to $J_2$ one crosses at least one inner boundary point $x_k$ with $1\le k\le n-1$. At such a point, one of the letters changes. By the monotonicity just proved, this change cannot later be undone when crossing further inner boundary points. Hence the return words on $J_1$ and $J_2$ are different.
\end{proof}
The strategy for the proof of Theorem~\ref{theorem:nodirichlet} is to vary the endpoint $a_n$ to enlarge the interval $G=[a_n,a_0)$ until we reach a limiting case $\tilde{G}=[a_{n-1},a_0)$ with respect to which the return words either contain one less letter or which has length greater than $\|\alpha\|$ (i.e. one of its return words is empty). We study how the collection of return words changes along such a trajectory with the conclusion is that if the return words of $G$ have a common Dirichlet eigenvalue then the same is true for $\tilde{G}$. The proof by induction ends once one has enlarged $G$ so much that only one letter is left or once one return word is empty.

We fix throughout points $0<a_0<...<a_n<1$, no two of which are on the same $\alpha$-orbit. From now we also assume as in Theorem~\ref{theorem:nodirichlet} that $|G|<\alpha$ (otherwise one of the return words would be empty). For $t\in (a_{n-1}, a_n)$ we then consider the perturbed boundary points $a_0(t),...,a_n(t)$ with $$a_k(t)=a_k, \quad  k\neq n$$
and
$$a_n(t)=t,$$ 
i.e. we are considering the interval $G_t=[a_n(t),a_0(t))=[t,a_0)$ which has return words $W_t = \{w_{G_t}(x): x\in G_t\}$ with respect to the intervals with endpoints $a_0,...,a_{n-1}$ and $a_n(t)=t$. 

As no two elements of $\{a_0,...,a_n\}$ are on the same $\alpha$-orbit the same will be true for $\{a_0(t),...,a_n(t)\}$ for all but countably many $t$. Moreover, only for finitely many $t$ there exists an algebraic coincidence where one can reach some $a_k(t)$ from $t$ in a number of steps less than the return time. We can describe exactly those exceptional points:

\begin{definition}
We define
$$Y= (a_{n-1}, a_n) \cap \{a_k-j\alpha: \, k=0,...,n-1, \quad j=-(r-1),\ldots,0,...,r-1\} $$
with $r$ the largest return time of $G$. 
We write the points of $Y$ in decreasing order as
\[
a_n>y_1>\dots>y_M>a_{n-1}.
\]
\end{definition}

\begin{lemma}
\label{lemma:xk_perturbation}
For $t\in (a_{n-1},a_n)$ consider the word boundaries
$$x_k(t) = a_k(t) - j_k(t)\alpha$$
as in Lemma~\ref{lemma:connected_component_words} with $j_k(t)$ the minimal number such that $x_k(t)\in [a_n(t), a_0(t))=[t,a_0)$. 

For all $s,t\in (y_i,y_{i-1})$ one has
$$x_k(s)=x_k(t),$$
if $k\neq n$ and
$$x_n(t) = x_n(s) + (t-s).$$
\end{lemma}
\begin{proof}
For $k=0,...,n$ and any jump discontinuity $t$ of $j_k$ one has
$$j_k(t-\epsilon)< j_k(t+\epsilon)$$
for every small enough $\epsilon>0$ since $j_k(t)$ is monotone increasing with $t$. By definition of $j_k$ one then has
$$a_k(t) - j_k(t-\epsilon)\alpha  = a_n(t)$$
or
$$a_k(t) - j_k(t-\epsilon)\alpha  = a_0$$
since $a_k(t+\epsilon)-j_k(t-\epsilon)\alpha \notin [a_n(t+\epsilon),a_0)$ for every small $\epsilon>0$. For $k=n$ only the second equation has a solution, namely when
$$t - j_k(t-\epsilon)\alpha  = a_0$$
hence $t\in Y$ since $j_k(t-\epsilon)< r$. For $k\neq n$ one has
$$a_k - j_k(t-\epsilon)\alpha  = t$$
or
$$a_k - j_k(t-\epsilon)\alpha  = a_0$$
and only the first equation has a solution, since $a_k$ and $a_0$ are not on the same $\alpha$-orbit. One concludes also $t\in Y$.

We therefore found that since the interval $(y_i,y_{i-1})$ does not intersect $Y$ all $j_k(t)$ are constant inside.
\end{proof}

\begin{lemma}
\label{lemma:continuity}
For $t\in [a_{n-1}, a_n)\setminus Y$ and any $x\in G_t\setminus \{x_0(t),...,x_n(t)\}$ one has
$$w_{G_s}(x)=w_{G_t}(x)$$
for all $s\in [a_{n-1}, a_n)$ with $$|s-t|<  \min_{k=0,...,n} \mathrm{dist}(x,x_k(t)).$$
\end{lemma}
\begin{proof}
For any  $x\in G_t\setminus \{x_0(t),...,x_n(t)\}$ we set 
$$\delta :=\min_{k=0,...,n} \mathrm{dist}(x,x_k(t))$$
and note that $(x-\delta, x+\delta)$ is contained entirely inside a single connected component of $G_t\setminus \{x_0(t),...,x_n(t)\}$. 

By Lemma~\ref{lemma:connected_component_words} the $j$-th letter of $w_{G_t}(y)$ is the same for all $y\in (x-\delta, x+\delta)$ and therefore none of the arcs $(x+j\alpha-\delta, x+j\alpha+\delta)$ for $0 < j <\ell_{G_t}(x)$ can contain any of the points $a_k(t)$. Equivalently this is a lower bound for the minimum distance 
$$\min_{k=0,...,n} \;\;\min_{j=1,...,\ell_{G_t}(x)-1} \mathrm{dist}(x+j\alpha, a_k(t))\geq \delta$$
that the return orbit of $x$ has from any of the boundary points $a_k(t)$.

For all $s$ with $|t-s|< \delta$ one therefore has $w_{G_s}(x)=w_{G_t}(x)$ as only the boundary point $a_n(s)$ is different and it moves by the distance $|a_n(s)-a_n(t)|=|t-s|$ which is not enough to change any letter. 
\end{proof}
Between transition points the return words cannot change. For a simple proof we assume the intervals $G_t$ do not get too large:
\begin{lemma}
\label{lemma:wordsconstant}
For any $t,s \in (y_i,y_{i-1})$ with $|G_t|<1-\|\alpha\|$ and $|G_s|<1-\|\alpha\|$ one has $W_s=W_t$.
\end{lemma}
\begin{proof}
We claim first that the points $x_0(t),...,x_n(t)$ are all distinct from each other for $t\in (y_i,y_{i-1})$ and also distinct from $t$, hence $W_t$ has exactly $n+2$ elements. Indeed, for $j,k\neq n$ this is immediate by the assumption that $a_j$ and $a_k$ have disjoint $\alpha$-orbits. If $$a_j-m\alpha=x_j(t)=x_n(t)=t-l\alpha$$ for some $1\leq m,l\leq r$, then $t=a_j-(m-l)\alpha$, where $-(r-1)\leq m-l\leq r-1$, which contradicts $t\notin Y$.

By Lemma~\ref{lemma:continuity} one can pick for every return word $W_t$ a representative $x\in G_t\setminus \{x_0(t),...,x_n(t)\}$ which yields the same return for $w_{G_s}(x)$ for all $s$ in an open neighborhood of $t$. Taking the intersection over $n+2$ such representatives gives a neighborhood $U$ of $t$ such that $W_s=W_t$ for all $s\in U$, hence $W_t$ is locally constant in $t$. 
\end{proof}

\begin{lemma}
\label{lemma:onewordchanges}
Fix some $y\in Y$ and set $W_{+}=W_{y+\epsilon}$ and $W_{-}=W_{y-\epsilon}$ for any small enough $\epsilon>0$. Then $|W_-\setminus W_+|\leq 1$.
\end{lemma}
\begin{proof}
Let $w\in W_+\setminus W_-$ be a pre-transition word which disappears. For every $t>y$ let $J_t\subset G_t$ be the maximal open interval such that $w_{G_t}(x)=w$ for all $x\in J_t$. The boundary points of $J_t$ are in $\{x_0(t),...,x_n(t)\}$, hence by applying Lemma~\ref{lemma:continuity} to the midpoint of $J_t$ one has $|J_t|\leq 2|t-y|$: if conversely one had $|J_t|>2|t-y|$ then the midpoint of $J_t$ would still represent the same return word and thus $w$ would have to be a return word in $W_{y-\epsilon}$ for some small enough $\epsilon$.

The points $x_0(t),...,x_n(t)$ are mutually distinct for $t>y$ and by Lemma~\ref{lemma:xk_perturbation} only $x_n(t)$ is not constant, therefore there can only be one unique index $k\neq n$ such that $$\lim_{t\to y+} x_n(t) = \lim_{t\to y+}x_k(t).$$
For $|J_t|$ to shrink to $0$ the interval $J_t$ can only be of the form $J_t=(x_k(t), x_n(t))$ for $t\to y+$. As $k$ is unique and $|W_+|=|W_-|=n+2$ we conclude there can be at most one word in $W_+\setminus W_-$ and therefore $W_-\setminus W_+$ also contains at most one element.
\end{proof}

\begin{lemma}
\label{lemma:wordchangeanatomy}
Fix some $y\in Y$ and for any small enough $\epsilon>0$ set $W_{+}=W_{y+\epsilon}$ and $W_{-}=W_{y-\epsilon}$. Then for every word \(\tilde w\in W_-\setminus W_+\) there exist words \(w\in W_+\), \(\tilde v\in W_-\), and a letter \(l\) such that
\[
w=\tilde w\,l\,\tilde v
\qquad\text{or}\qquad
w=\tilde v\,l\,\tilde w.
\]
\end{lemma}
\begin{proof}
Fix \(\varepsilon>0\) so small that $y$ is the only point of $Y$ in $
(y-\varepsilon,y+\varepsilon)$
		and such that the smallest return time of any interval of size $2\varepsilon$ to itself is larger than the maximal return times of the intervals $G^\pm=[y\pm\epsilon, a_0)$ to themselves.

Let \(\tilde w\in W_-\setminus W_+\), and choose \(x\in G^-\) such that
\[
w_{G^-}(x)=\tilde w.
\]
The \(G^-\)-return orbit of \(x\) must intersect $G^-\setminus G^+$, for otherwise $\tilde w$ would define a return word already belonging to \(W_+\). Moreover, the \(G^-\)-return orbit of \(x\) also cannot both start \emph{and} end in $G^-\setminus G^+$ since we assumed via choice of $\varepsilon$ that the return time of any point in $G^-\setminus G^+=[y-\varepsilon,y+\varepsilon)$ to $G^-\setminus G^+$ itself is larger than the return time to $G^-$. Therefore, either the starting point $x$ or the endpoint $x+\ell_{G^-}(x)\alpha $ of the return orbit is in $G^-\setminus G^+$ and the other one is in $G^+$. 

\begin{enumerate}
\item[Case 1:]
\[
x\in G^-\setminus G^+,
\qquad
			x+\ell_{G^-}(x)\alpha \in G^+.
\]
By Lemma~\ref{lemma:partition} there exists a unique point \(\tilde{x}\in G^+\) whose orbit meets \(x\) before returning to \(G^+\).
The \(G^+\)-return orbit of \(\tilde{x}\) reaches $x$ after some $j$ steps $x=\tilde{x}+j\alpha$, which is the first and only time it intersects $G^-\setminus G^+$. After that it follows the \(G^-\)-return orbit of \(x\). Hence
\[
w_{G^+}(\tilde{x})=\,\tilde v l \tilde w \in W_+
\]
where $\tilde v:=w_{G^-}(\tilde{x})\in W_-$ and $l$ is the $j$-th letter of $w_{G^+}(\tilde{x})$.
\item[Case 2:]
\[
x\in G^+,
\qquad
\tilde x:=x+\ell_{G^-}(x)\alpha \in G^-\setminus G^+.
\]
This is the same as case 1 except with the roles of $x$ and $\tilde{x}$ exchanged, i.e.
\[
w_{G^+}(x)=\,w_{G^-}(x) l w_{G^-}(\tilde{x}) \in W_+.
\]
\end{enumerate}
In both cases we have produced a word \(w \in W_+\) such that
\[
w=\tilde w\,l\,\tilde v
\qquad\text{or}\qquad
w=\tilde v\,l\,\tilde w
\]
with some word $\tilde{v}\in W_-$. It remains to show that if $\tilde{w}\in W_-\setminus W_+$ then one must have $\tilde{v}\in W_+$. By Lemma~\ref{lemma:onewordchanges} the set $W_-\setminus W_+$ contains at most one element, hence it is sufficient to prove that $\tilde{w}\neq \tilde{v}$ are distinct.

As we argued in the proof of Lemma~\ref{lemma:onewordchanges} the maximal interval $J\subset G^-$ that the post-transition word $\tilde{w}\in W_-\setminus W_+$ occurs on has size $|J|\leq 2\varepsilon$, since otherwise it would also be in $W_+$ by continuity. 
By Lemma~\ref{lemma:connected_component_words} the only way for $\tilde{w}$ and $\tilde{v}$ to be equal is for the points $x,\tilde{x}$  that give rise to them to lie in that same connected component $J\subset G^-\setminus \{x_0(G^-),\dots,x_n(G^-)\}$. By the construction above, either $x$ or $\tilde{x}$ is in $G^+$ and the other of the two points is on its $G^+$-return orbit. Since $x,\tilde{x}\in J$ would imply $|x-\tilde{x}|\leq 2\varepsilon$ this produces a contradiction, as by assumption on $\varepsilon$ the first return time of any interval of size $2\varepsilon$ to itself is longer than $\ell_{G^+}(x)$ and $\ell_{G^+}(\tilde{x})$. Hence $\tilde{w}$ and $\tilde{v}$ are distinct, which completes the proof.

\end{proof}

It remains to combine all ingredients:
\begin{proof}(of Theorem~\ref{theorem:nodirichlet})
We proceed by induction over $n$. If $n=1$ then all return words are composed of only one letter and Proposition~\ref{proposition:no_common_ev_dirichlet_laplacian} finishes the proof.

		Assume therefore $n\geq 2$ and let us be given $0<a_0<\cdots<a_n<1$ such that no two of these points lie on the same $\alpha$-orbit. For an eventual contradiction we further assume that the return words $W$ of $G=[a_n,a_0)$ with respect to this $(n+1)$-interval configuration have a common Dirichlet eigenvalue. We may assume $|G|<\|\alpha\|$ as otherwise some return word is empty and we rule an empty word as no common Dirichlet eigenvalue. 

		We now increase the size of $G$. Along the path $t\in (a_{n-1},a_n]$ we have the possible finitely many transition points $Y=\{y_1,....,y_M\}$. By Lemma~\ref{lemma:wordsconstant} the collection of return words $W_{t}$ is locally constant for $t$ outside $Y$. Pick one representative from each component of $(a_{n-1},a_n)\setminus Y$ and label them so that $a_{n-1}<t_0<\cdots<t_M<a_n$.

Assume $W_{t_{i+1}}$ has a common Dirichlet eigenvalue, which holds for the base case $W_{t_M}=W$. By Lemma~\ref{lemma:onewordchanges} and Lemma~\ref{lemma:wordchangeanatomy} one has $|W_{t_i}\setminus W_{t_{i+1}}|\leq 1$ and if there is some word that changes then Lemma~\ref{lemma:stripping} implies that $W_{t_i}$ also has a common Dirichlet eigenvalue. 

Note that in particular $W_{t_{i}}$ cannot contain an empty word as we rule an empty word as no common Dirichlet eigenvalue, which means the lengths of the intervals $G_{t_i}$ stay bounded by $\|\alpha\|$. If there was no contradiction up until $t_0$ then inductively $W_{t_0}$ has a common Dirichlet eigenvalue. 

		It remains to compare $W_{t_0}$ to the return words $\tilde{W}$ of $\tilde{G}=[a_{n-1},a_0)$ with respect to the $n$ intervals with endpoints $a_0,...,a_{n-1}$. As there are no more transition points between $a_{n-1}$ and $t_0$ the interval $\tilde{G}$ must still have length less than $\|\alpha\|$ as the interval with length exactly $\|\alpha\|$ can only occur as a transition point. By Lemma~\ref{lemma:connected_component_words} the collections $\tilde{W}$ and $W_{t_0}$ therefore have exactly $n+1$ respectively $n+2$ elements. There must be at least one word $w\in W_{t_0}\setminus \tilde{W}$ that disappears in the limit of $t\to a_{n-1}$ and precisely as in the proof of Lemma~\ref{lemma:onewordchanges} that there is exactly one return word in $W_{t_0}$ whose representing interval shrinks to zero, while the others stay finite. Therefore, $\tilde{W}$ is a subset of $W_{t_0}$, which accordingly inherits the common Dirichlet eigenvalue. However, the induction hypothesis was that for any $n$-interval configuration the collection of return words has no common Dirichlet eigenvalue, which finishes the proof by contradiction.
\end{proof}
\appendix
\section{Monotone eigenvalue branches}
Fix a $1$-periodic function \mbox{$f:\mathbb{R}\to\mathbb{R}$} and let $\alpha\in\mathbb{R}\setminus\mathbb{Q}$. We consider the family of discrete Schr\"odinger operators on $\ell^2(\mathbb{Z})$,
\begin{equation}\label{eq:operator_family}
(H(x)\psi)(m)=\psi(m+1)+\psi(m-1)+f(x+m\alpha)\,\psi(m),\qquad x\in[0,1).
\end{equation}
For $n\in\mathbb{N}$, we denote $H(x)_{[0,n-1]}=H_n(x)$. Denote the (simple) eigenvalues of $H_n(x)$ in increasing order by
\begin{equation*}
-\infty < E_0(x)<E_1(x)<\cdots<E_{n-1}(x).
\end{equation*}
Let us define
\begin{equation*}
0=\beta_0<\beta_1<\cdots<\beta_{n-1},
\qquad
\{\beta_0,\dots,\beta_{n-1}\}=\bigl\{\{-j\alpha\}:j=0,1,\dots,n-1\bigr\}.
\end{equation*}
	Note that for every $0\leq k\leq n-1$, $E_k(x)$ is 1-periodic. We will set $\beta_n=1+\beta_0$ and consider the arcs $[\beta_j, \beta_{j+1})$ which implicitly includes the wrap-around. Under $\mathcal{A}_1$, the restrictions $E_k|_{[\beta_j,\beta_{j+1})}$ are monotone for $n\geq n_0$. We have the following (see also \cite{JK2,Kach}):

\begin{proposition}\label{prop:Lambda_functions}
    There exist functions $\Lambda_k:[0,1)\to\mathbb{R}$, $k=0,1,\ldots,n-1$, such that
    \begin{equation}
        \sigma\bigl(H_n(x)\bigr) = \Bigl\{\Lambda_0(x),\ \Lambda_1\bigl(\{x-\beta_{n-1}\}\bigr),\ \dots,\ \Lambda_{n-1}\bigl(\{x-\beta_1\}\bigr)\Bigr\}, \label{eq:spec-decomp}
    \end{equation}
    and
    \begin{equation}
			E_k(0) = \Lambda_k\bigl(\{-\beta_{n-k}\}\bigr),\qquad k=0,1,\dots,n-1. \label{eq:normalization}
    \end{equation}
		Furthermore, for $n\geq n_0$ and every $0\leq k\leq n-1$, $\Lambda_k$ is monotone on $[0,1)$.
\end{proposition}
\begin{proof}
		Define $\Lambda_k(x)=E_{(j+k)\bmod{n}}(\{x+\beta_{n-k}\})$ for $\{x+\beta_{n-k}\}\in[\beta_j,\beta_{j+1})$. Clearly, $\Lambda_k(\{x-\beta_{n-k}\})=E_{(j+k)\bmod n}(x)$ which for fixed $x$ runs through all $n$ eigenvalues proving \eqref{eq:spec-decomp}. Monotonicity follows from $\mathcal{A}_1$ within each piece and $\mathcal{A}_2$ at changes of index.
\end{proof}
	It will be more useful to rephrase $\mathcal{A}_1$ in terms of the functions $(\Lambda_k)_{k=0}^{n-1}$. Hence, throughout we assume the following.

\begin{assumption}\label{Assumption1}
There exists $\gamma>0$ and $n_0\in \NM$ such that for $n\geq n_0$, the functions $(\Lambda_k)_{k=0}^{n-1}$ given by Proposition \ref{prop:Lambda_functions} satisfy
\begin{equation}\label{eq:gamma-monotone}
\Lambda_k(y)-\Lambda_k(x) \ge \gamma\,(y-x)\quad\text{for all }0\le x\le y<1.
\end{equation}
\end{assumption}

\begin{assumption}\label{Assumption2}

For every $x\in[0,1)$ and every $k=1,\dots,n-1$,
\begin{equation}\label{eq:interlace}
E_k(x) \ge E_{k-1}(x-0).
\end{equation}
\end{assumption}
\begin{remark}
    It is immediate from the definition of $(\Lambda_k)_{k=0}^{n-1}$ that Assumptions 1 and 2 are equivalent to $\mathcal{A}_1$ and $\mathcal{A}_2$ given in the introduction.
\end{remark}
	In this section we prove basic results on the integrated density of states (IDS) $N(E)$ and the eigenvalue counting functions $N_n(x,E)=\#\{j:E_j(x)\leq E\}$ for the restrictions to $[0,n-1]$. Write $\overline N_n(E)=\int_0^1N_n(x,E)\,dx$.

A direct consequence of the anti-Lipschitz property is the Lipschitz continuity of the IDS, namely Lemma \ref{lemma:ac_IDS}.
\begin{proof}[Proof of Lemma \ref{lemma:ac_IDS}]
Following \cite{Kach}, for $E_1<E_2$ we have
\begin{equation}
			\frac{1}{n}\overline N_n(E_2)-\frac{1}{n}\overline N_n(E_1)=\frac{1}{n}\sum_{i=0}^{n-1}\int_0^1\chi_{(E_1,E_2]}\left[\Lambda_i(x)\right]dx\leq\gamma^{-1}(E_2-E_1),
\end{equation}    
		where the inequality follows from Assumption \ref{Assumption1}. Taking the limit for $n\to \infty$ completes the proof.
\end{proof}

Let us also mention the Thouless formula which we will need later in the appendix (see, e.g.\ \cite{Thou}):
\begin{equation}\label{eq:Thouless}
L(E)=\int_{\mathbb{R}}\log|E-E'|\,dN(E').
\end{equation}
The goal of the appendix is to present the technical details which lead to the regularity lemmas described in Section \ref{sec:setting}.

To simplify notations we will denote $q\coloneqq q_n$ whenever $n$ is clear from context.
\begin{proposition}\label{JK:3.4}
    For every $q\in\mathbb{N}$ a continued fraction denominator, $E\in\mathbb{R}$ and $x,y\in[0,1)$,
    \begin{equation}\label{eq:Nqx-Nqy}
        \left|N_{q}(x,E)-N_{q}(y,E)\right|\leq 2n_0+22.
    \end{equation}
    As a consequence, we have
     \begin{equation}\label{eq:discrepancy_dos}
        \left|\frac{1}{q}N_{q}(x,E)-N(E)\right|<\frac{4n_0+50}{q}.
    \end{equation}
\end{proposition}

\begin{lemma}
\label{lemma:eigenvalue_monotonicity_tech1}
If $\{q\alpha\}>\frac{1}{2}$ one has
\begin{equation}\label{eq:small_qa}
    N_{q}(x,E) - 6-n_0 \leq N_{q}(x+m\alpha,E) \leq N_{q}(x+ q\alpha,E) + 6+n_0
\end{equation}

and if $\{q\alpha\}<\frac{1}{2}$, one has
\begin{equation}\label{eq:large_qa}
    N_{q}(x+q\alpha,E) - 6-n_0 \leq N_{q}(x+m\alpha,E) \leq N_{q}(x,E) + 6+n_0
\end{equation}
for all $m \in 0,...,q-1$.
\end{lemma}
\begin{proof}
 Assume first $\{q\alpha\}<\frac{1}{2}$, in which case one actually has $\{q\alpha\}<\frac{1}{q}$. Fix $0\leq m\leq q-1$. There exist perturbations $R_1$, $R_2$ of rank at most $2$ and a unitary operator $U$ (a cyclic shift) such that
\begin{align*}
H_q(x+m\alpha) &= H_{q-m}(x+m\alpha)\oplus H_{m}(x+q\alpha) + R_1 \\
H_{q-m}(x+m\alpha)\oplus H_{m}(x)&= U H_{q}(x)U^* +R_2.
\end{align*}
		If $m\geq n_0$ then the eigenvalue counting function $y\mapsto N_{m}(y,E)$ is non-increasing except for at jump points $\beta_j$ which are on the arc $[x, x+q\alpha]$ where it can jump by $1$ each. Since $q$ is a continued fraction denominator the distance between consecutive jump points satisfies $|\beta_{j+1}-\beta_{j}|>\frac{1}{2q}$ (see e.g. \cite[Proposition 2.1]{JK2}), hence there can be at most two points on the arc. We then conclude
\begin{align*}
N_q(x+m\alpha,E)&\leq N_{q-m}(x+m\alpha,E) + N_{m}(x+q\alpha,E) +2 \\
&\leq N_{q-m}(x+m\alpha,E) + N_{m}(x,E) +4 \\
&\leq N_{q}(x,E) + 6
\end{align*}
If $m<n_0$ we can bound trivially $$N_{m}(x+q\alpha,E)\leq N_{m}(x,E) + n_0$$
hence we can use the bound $N_{m}(x+q\alpha,E)\leq N_{m}(x,E) + 6+n_0$ in either case. Replacing $x$ by $x+m\alpha$ and $m$ by $q-m$ gives the inequality in the other direction.    
\end{proof}

\begin{proof}[Proof of Proposition \ref{JK:3.4}]
Fix $E\in\mathbb{R}$ and write $N(x)\coloneqq N_q(x,E)$. Set
\[
D(x)\coloneqq N(x+q\alpha)-N(x),\qquad x\in[0,1).
\]
By shift invariance of Lebesgue measure,
\[
\int_{[0,1)} D(x)\,dx = \int_{[0,1)}\bigl(N(x+q\alpha)-N(x)\bigr)\,dx=0.
\]
		Moreover, $D$ is integer-valued and piecewise constant. Suppose $\{q\alpha\}< \frac{1}{2}$, then $D(x)\leq 1$ as $N$ has at most one positive jump at some $\beta_k$ along the arc $[x,x+q\alpha]$ and at all other discontinuities $N$ only decreases. Conversely, if $\{q\alpha\}> \frac{1}{2}$ then $D(x)\geq -1$. For the average to be $0$ there must hence exist in either case some $x_0\in[0,1)$
such that
\begin{equation}\label{eq:Dsmall}
|D(x_0)|\le 1,
\qquad\text{equivalently}\qquad
|N(x_0+q\alpha)-N(x_0)|\le 1.
\end{equation}

Combining \eqref{eq:small_qa} \eqref{eq:large_qa} with \eqref{eq:Dsmall}, we obtain that for every $0\le m\le q-1$,
\begin{equation}\label{eq:orbitbound}
    N(x_0)-7-n_0\le N(x_0+m\alpha)\le N(x_0)+7+n_0.
\end{equation}

Let $\beta_0<\beta_1<\cdots<\beta_{q-1}$ be the reset points defined above and set $\beta_q=\beta_0+1$.
Write $I_j=[\beta_j,\beta_{j+1})$, with indices modulo $q$. Recall that $N|_{(\beta_j,\beta_{j+1})}$ is monotone
decreasing, and when crossing $\beta_{j+1}$ it can increase by at most $1$.

For continued fraction denominators points on the orbit $\mathcal{O}(x_0)=\{x_0+m\alpha:0\le m\le q-1\}$ have maximum gap $2/{q}$ and also $|I_j|\geq 1/{2q}$. Hence for each $j$ we can find
\[
y_j^-\in \mathcal{O}(x_0)\cap (I_{j-4}\cup I_{j-3}\cup I_{j-2}\cup I_{j-1}),
\qquad
y_j^+\in \mathcal{O}(x_0)\cap (I_{j+1}\cup I_{j+2}\cup I_{j+3}\cup I_{j+4}).
\]
Fix $x\in I_j$. Moving from $y_j^-$ to $x$ crosses at most $4$ jump points, hence $N(x)\le N(y_j^-)+4$.
Similarly, moving from $x$ to $y_j^+$ crosses at most $4$ jump points, hence $N(y_j^+)\le N(x)+4$, that is
$N(x)\ge N(y_j^+)-4$. Together with \eqref{eq:orbitbound} this yields
\[
N(x_0)-11-n_0\le N(x)\le N(x_0)+11+n_0,\qquad x\in I_j.
\]
Since $j$ was arbitrary, the same holds for all $x\in[0,1)$, and therefore for all $x,y\in[0,1)$,
\[
|N(x)-N(y)|\le 22+2n_0,
\]
which is precisely \eqref{eq:Nqx-Nqy}. Now, \eqref{eq:discrepancy_dos} follows by applying the triangle inequality to the standard estimate
$$\left|N(E)- \frac{1}{q}\int_\TM N_q(x,E)\mathrm{d}x\right|\leq \frac{2}{q}.$$
\end{proof}

\begin{corollary}\label{cor:JK-3.5}
    Let $m\in\mathbb{N}$. Suppose that $m=sq+r$ with $q,m\geq n_0$. Then for every $x,y\in\mathbb{R}$,
    \begin{equation}\label{eq:JK-3.7}
        \left|N_m(x,E)-sN_q(y,E)\right|\leq (2n_0+22)s+4(s+|r|)
    \end{equation}
    and as a consequence
    \begin{equation}\label{eq:JK-3.8}
        \left|N_m(x,E)-N_m(y,E)\right|\leq (4n_0+44)s+8(s+|r|).
    \end{equation}
\end{corollary}
\begin{proof}
    Any Schr\"{o}dinger operator acting on $\ell^2(I)$, where $I$ is some integer interval, can be disconnected along an edge using a perturbation of rank at most $2$. Hence, by the structure of $H_m(x)$, we get that
    \begin{equation}
    \label{eq:block_decomposition} 
        H_m(x)=H_q(x)\oplus H_q(x+q\alpha)\oplus\ldots\oplus H_q(x+(s-1)q\alpha)+P
    \end{equation}
    where $\operatorname{rank}P\leq 2s+|r|$ (if $r<0$ this is to be understood as an equality after padding $H_m(x)$ with $r$ zero diagonal entries). Indeed, using rank $2$ perturbations, we disconnect along the edges $(kq-1,kq)$ for $1\leq k\leq s$, and subtract the remaining part $H_r(x+sq\alpha)$ which is of rank at most $r$. Now we apply Proposition \ref{JK:3.4} for each direct summand and get the result. \eqref{eq:JK-3.8} now follows from the triangle inequality.
\end{proof}
For the rest of the appendix, whenever we have $n=sq+r$, we assume that $|r|\leq\sqrt{n}$ and $q\geq n_0$. In particular, this implies that $|r|\leq\frac{sq}{2}$ whenever $n\geq 9$.
\begin{corollary}
With the same notation, we have that
\begin{equation}\label{eq:JK-3.12}
    \left|\frac{1}{m}N_m(x,E)-N(E)\right|\leq\frac{8|r|}{sq}+\frac{8n_0+200}{q}.
\end{equation}
    As a consequence, $\frac{1}{n}N_n(x,E)$ converges to $N(E)$ uniformly in $x$.
\end{corollary}
\begin{proof}
    A direct consequence of Corollary \ref{cor:JK-3.5} and the triangle inequality is that for $m=sq+r$ we have
\begin{align}
\left|\frac{1}{m}N_m(x,E)-\frac{1}{q}N_q(x,E)\right|
&\leq\left| \frac{1}{m}N_m(x,E)-\frac{s}{m}N_q(x,E)\right|
    +\left|\frac{s}{m}N_q(x,E)-\frac{1}{q}N_q(x,E)\right| \notag\\
&\leq\left(\frac{(2n_0+26)s}{m}+\frac{4|r|}{m}\right)
    +\left|\frac{s}{m}-\frac{1}{q}\right|N_q(x,E) \notag\\
&\leq\frac{4n_0+52}{q}+\frac{8|r|}{sq}.
\end{align}
By the triangle inequality, for $m_1=s_1q+r_1$ and $m_2=s_2q+r_2$ we now have
    \begin{equation}
        \left|\frac{1}{m_1}N_{m_1}(x,E)-\frac{1}{m_2}N_{m_2}(y,E)\right|\leq \frac{8|r_1|}{s_1q}+\frac{8|r_2|}{s_2q}+\frac{8n_0+200}{q}.
    \end{equation}
    Taking $m=sq+r$, $m'=s'q$ for some $m'\in\mathbb{N}$, we obtain
    \begin{equation}
        \left|\frac{1}{m}N_{m}(x,E)-\frac{1}{m'}N_{m'}(y,E)\right|\leq \frac{8|r|}{sq}+\frac{8n_0+200}{q}.
    \end{equation}
    Integrating this inequality over $y$ and taking $m'$ to $\infty$ yields \eqref{eq:JK-3.12}, which immediately implies uniform convergence.
\end{proof}

\begin{corollary}\label{cor:JK-3.7}
    Let $g:\mathbb{R}\to\mathbb{R}$ be a function of bounded variation. Then for every $x\in\mathbb{R}$,
    \begin{equation}
        \left|\int_\mathbb{R}g(E')dN(E')-\frac{1}{n}\sum\limits_{j=0}^{n-1}g(E_j(x))\right|\leq\|g\|_{\text{BV}}\left(\frac{8|r|}{sq}+\frac{8n_0+200}{q}\right).
    \end{equation}
\end{corollary}
\begin{definition}
    The discrepancy of a sequence $z_1,\ldots,z_n\in[0,1]$ is defined by
    \begin{equation*}
        D_n^*(z_1,\ldots,z_n)=\sup\limits_{t\in(0,1]}\left|\frac{\#\left\{k:z_k\in[0,t)\right\}}{n}-t\right|.
    \end{equation*}
\end{definition}
If $q$ is a continued fraction denominator then one has
$$D_q^*(x,x+\alpha,\ldots,x+(q-1)\alpha) \leq \frac{2}{q}$$
where it is understood that we consider the fractional parts only. By iterating the triangle inequality
$$D_{n+m}^*(x_1,\ldots,x_n, y_1,...,y_m)\leq \frac{n}{n+m}D_{n}^*(x_1,\ldots,x_n) + \frac{m}{n+m}D_{m}^*(y_1,...,y_m)$$
one can estimate the discrepancy for general $n=sq+r$ by
\begin{equation}
\label{eq:orbit_discrepancy}D_n^*(x,x+\alpha,\ldots,x+(n-1)\alpha) \leq \frac{4(s+|r|)}{n}.
\end{equation}

\begin{proposition}[{Koksma's Inequality \cite[Theorem 2.5.1]{KuipNied}}]\label{prop:Koksma}
    Let \mbox{$f:[0,1]\to\mathbb{R}$} be a function of bounded variation and let $z_1,\ldots,z_n$ be a sequence of points in $[0,1]$. Then
    \begin{equation}
        \left|\int_0^1f(t)\,dt-\frac{1}{n}\sum\limits_{i=1}^n f(z_i)\right|\leq\|f\|_{\text{BV}}\cdot D_n^*(z_1,\ldots,z_n).
    \end{equation}
\end{proposition}
\begin{proof}[Proof of Corollary \ref{cor:JK-3.7}]
With the empirical eigenvalue distribution
\[
\nu_{n,x}:=\frac1n\sum_{j=0}^{n-1}\delta_{E_j(x)}
\]
and the DOS measure $\nu$ we can conclude from \eqref{eq:JK-3.12} that
\[
\sup_{E\in\mathbb R}
\left|\nu_{n,x}((-\infty,E])-\nu((-\infty,E])\right|
\leq \frac{8|r|}{sq}+\frac{8n_0+200}{q}.
\]
By a limiting argument the same bound holds for open
half-lines. As $\nu_{n,x}$ and $\nu$ are probability measures we conclude the same bound also holds half-lines $[E,\infty)$. As both measures have total mass one,
the bound also holds for upper half-lines.

For nondecreasing functions $g$ the estimate is therefore immediate from the layer-cake representation $$g(E)- a= \int_{a}^b 1_{\{g(E)>e\}} de$$
with $a=g(-\infty)$ and $b=g(\infty)$. As $g$ is increasing $\{E: g(E)>e\}$ is a half-line, hence
$$\left|\int_{\mathbb R}g\,d\nu
-\frac1n\sum_{j=0}^{n-1}g(E_j(x))\right| \leq (b-a) \left(\frac{8|r|}{sq}+\frac{8n_0+200}{q}\right).$$
For the general case, one can decompose $g=c+ g_+- g_-$ with a constant $c$ and the difference of two nondecreasing functions for which
$$\mathrm{Var}(g) = g_+(\infty)-g_+(-\infty) + g_-(\infty)-g_-(-\infty).$$
\end{proof}

\section{Large deviation estimates for determinants}\label{AppendixB}
Throughout, we fix parameters $0<B_-< 1 <10<B_+<\infty$ and assume $n\geq n_0$.

Introduce the notation $|E-E'|_{[B_-,B_+]}$ for the truncated absolute value function
$$|E-E'|_{[B_-,B_+]} = \begin{cases}
1 & \text{if } |E-E'|<B_-\\
|E-E'| & \text{if } |E-E'|\in [B_-,B_+]\\
1 & \text{if } |E-E'|>B_+\\
\end{cases},$$
and let
\[
L(E;[B_-,B_+]) := \int_{\mathbb{R}} \log \lvert E - E' \rvert_{[B_-,B_+]} \, dN(E')
= L(E;[B_-,1]) + L(E;[1,B_+]).
\]
We will need the following
\begin{lemma}[{\cite[Lemma 2.2]{JK2}}]\label{lemma:lyapunov_uniformity}
    We have the following:
    \begin{equation}
        L(E)=(1+o(1))\log |E|,\qquad E\to\pm\infty,
    \end{equation}
    and 
    \begin{equation}
        \bigl|L(E;[B_-,B_+^{\log(2+\lvert E\rvert)}]) - L(E)\bigr|
\leq o(1)(1+L(E)),
\qquad \text{as } B_- \to +0,\; B_+ \to +\infty,
    \end{equation}
    where $o(1)$ is uniform in $E$.
\end{lemma}

Let
\[
    P_n(x,E)\coloneqq\det (H_n(x)-E).
\]
Then one can decompose
\begin{equation}
    P_n(x,E)=P_n^{\text{small}}(x,E)P_n^{\text{mid}}(x,E)P_n^{\text{large}}(x,E).
\end{equation}
with
\begin{equation}
    P_n^{\text{small}}(x,E)=\prod\limits_{k:|E_k(x)-E|<B_-} (E_k(x)-E),
\end{equation}
\begin{equation}\label{eq:pn_large_def}
    P_n^{\text{large}}(x,E)=\prod\limits_{k:|E_k(x)-E|>B_+^{\log(2+|E|)}} (E_k(x)-E),
\end{equation}
and
\begin{equation}\label{eq:JK-4.4}
    P_n^{\text{mid}}(x,E)=\prod\limits_{k:B_-\leq |E_k(x)-E|\leq B_+^{\log(2+|E|)}} (E_k(x)-E).
\end{equation}

Following \cite{JK2}, we will now analyze the three parts $P_n^{\text{small}}$, $P_n^{\text{large}}$, and $P_n^{\text{mid}}$.
\subsection{Estimates for \texorpdfstring{$P_n^{\text{small}}$}{P_n^small}}
Given $0\leq j\leq n-1$, define the \textit{intersection points} $z_j(E)$ by
\begin{equation}\label{eq:intersection_pts}
\begin{split}
    t_j(E)&\coloneqq\inf\{x\in [0,1): \Lambda_j(x)\geq E\}, \qquad \inf\emptyset=1,\\
    z_j(E)&\coloneqq \{\beta_{n-j}+t_j(E)\}.
\end{split}
\end{equation}
Fix $0\leq j\leq n-1$ and denote $\Lambda_j(x-\beta_{n-j})\coloneqq \lambda(x)$. Since the function $\Lambda_j(x)$ is monotone increasing, $z_j(E)$ is the unique preimage of $E$ under $\lambda$ function if one exists. 
By Assumption \ref{Assumption1}, for $0\leq k\leq n-1$ we have
\begin{equation}\label{eq:JK-4.7}
\begin{aligned}
\left|\Lambda_k\bigl(\{x-\beta_{n-k}\}\bigr)-E\right|
&\geq \gamma \left|\{x-\beta_{n-k}\}-t_k(E)\right| \\
&\geq \gamma \|x-z_k(E)\| \\
&= \gamma \dist\bigl(x-z_k(E),\mathbb{Z}\bigr).
\end{aligned}
\end{equation}
\begin{proposition}
For any $n \in \mathbb{N}$ we have
\begin{equation}
\label{eq:JK-4.8}
    1\geq \left|P_n^{\text{small}}(x,E)\right|\geq\prod\limits_{k:\gamma\|x-z_k(E)\|<B_-}\gamma\|x-z_k(E)\|.
\end{equation}
If $n=sq+r$ for $q$ a denominator of the continued fraction expansion then the number of factors on the right-hand side can be bounded by
    \begin{equation}
        \#\left\{0\leq k\leq n-1:\gamma\|x-z_k(E)\|<B_-\right\}\leq\frac{4sq B_-}{\gamma}+(8n_0+100)(s+|r|).
    \end{equation}
\end{proposition}
\begin{proof}
By definition
$$1\geq \left|P_n^{\text{small}}(x,E)\right| = \prod\limits_{k:|\Lambda_k(\{x-\beta_{n-k}\})-E|<B_-}\left|\Lambda_k\bigl(\{x-\beta_{n-k}\}\bigr)-E\right|$$
with the inequality a result of $B_-<1$. Then \eqref{eq:JK-4.8} follows immediately from \eqref{eq:JK-4.7}.

For $x<y$ we have 
$$
N_n(y,E)-N_n(x,E)\leq \# \{\text{jump points in }[x,y]\}- \# \{k: z_k(E) \in (x,y)\}
$$
where by jump points we mean the endpoints of monotonicity intervals. 

One can bound the number of jump points by
$$\#\{\text{jump points in }[x,y]\}\leq s\lceil 2q (y-x)\rceil + |r|,$$
e.g. by noting that this is the maximum number of jump points of the right-hand side of \eqref{eq:block_decomposition}. The number of points $z_k(E)$ with \mbox{$\gamma\|x-z_k(E)\|<B_-$} can now be estimated by applying this bound together with \eqref{eq:JK-3.8} to the interval $[x-\frac{B_-}{\gamma},x+\frac{B_-}{\gamma}]$ which results in the stated inequality.
\end{proof}
\subsection{Estimates for \texorpdfstring{$P_n^{\text{large}}$}{P_n^large}}
Here we bound the high-energy contribution
$$\left|P_n^{\text{large}}(x,E)\right| = \prod\limits_{k:\left|\Lambda_k\bigl(x-\beta_{n-k}\bigr)-E\right|>B_+^{\log(2+|E|)}}\left|\Lambda_k\bigl(x-\beta_{n-k}\bigr)-E\right|.$$
By basic perturbation theory every eigenvalue of $H_n(x)$ is an eigenvalue of the potential $V_n(x)$ shifted by at most $\norm{\Delta}=2$. Hence 
\begin{equation}\label{eq:JK_4.10}
\begin{aligned}
\prod\limits_{k:\left|f(x+k\alpha)-E\right|>B_+^{\log(2+|E|)} + 2}\bigl(\left|f(x+k\alpha)-E\right|-2\bigr)
&\leq  \left|P_n^{\text{large}}(x,E)\right|\\
&\leq \prod\limits_{k:\left|f(x+k\alpha)-E\right|>B_+^{\log(2+|E|)} -2}\bigl(\left|f(x+k\alpha)-E\right|+2\bigr).
\end{aligned}
\end{equation}

The number of factors in this product can be bounded as follows:
\begin{proposition}\label{prop:num_of_large_entries}
For large enough $B_+$ we have
\begin{equation}\label{eq:JK_4.11}
    \#\Bigl\{k:\, \lvert f(x+k\alpha)-E\rvert > B_+^{\log(2+|E|)} - 2\Bigr\}\leq 8n_0(s+|r|)+n\frac{n_0 C(f)}{\log B_+},
\end{equation}
where
\[
    C(f) := \sup_E \frac{1}{\log(2+\lvert E\rvert)}
\int_0^1 \log\bigl(\lvert f(x)-E\rvert + 6\bigr)\,dx < \infty .
\]

\end{proposition}
\begin{proof}
Define
$$S=\{x\in [0,1): \|H_{n_0}(x)-E\|>B_+^{\log(2+|E|)}-2\}.$$
One has
$$|f(x)-E|\leq \|H_{n_0}(x)-E\|\leq \max_{0\leq j< n_0} |f(x+j\alpha)-E|+2$$
since $f(x)-E$ is a diagonal element of $H_{n_0}(x)-E$. This gives
$$\{|f-E|>B_+^{\log(2+|E|)}-2\}\subset S \subset \bigcup_{j=0}^{n_0-1} \{x: |f(x+j\alpha)-E|>B_+^{\log(2+|E|)}-4\}.$$
If one has $x\notin S$ then the largest and smallest eigenvalue of $H_{n_0}(x)$ respectively must satisfy
$$E_{\mathrm{min}}(x)\geq E-B_+^{\log(2+|E|)}+2, \qquad E_{\mathrm{max}}(x)\leq E+B_+^{\log(2+|E|)}-2.$$
Since the eigenvalues of $H_{n_0}$ are increasing outside $n_0$ of the reset points $\beta_j$, the set $S^c$ is a union of at most $n_0$ intervals. The characteristic functions $\chi_S$ therefore has variation $\|\chi_S\|_{\mathrm{BV}}\leq 2n_0$.

Using the Koksma inequality and \eqref{eq:orbit_discrepancy} we can bound 
$$
\begin{aligned}
&\#\bigl\{0\leq k<n:\,|f(x+k\alpha)-E|>B_+^{\log(2+|E|)}-2\bigr\}\\
&\leq 4(s+|r|)\|\chi_S\|_{\mathrm{BV}}+n|S|\\
&\leq 8n_0(s+|r|)
   +nn_0\bigl|\{x\in[0,1):\,|f(x)-E|>B_+^{\log(2+|E|)}-4\}\bigr|\\
&\leq 8n_0(s+|r|)
   +\frac{nn_0}{\log(2+|E|)\log B_+}
      \int_0^1\log(|f(u)-E|+6)\,du\\
&\leq 8n_0(s+|r|)
   +n\frac{n_0C(f)}{\log B_+}.
\end{aligned}
$$
Here we used translation-invariance of the Lebesgue measure and the Markov inequality. 

\end{proof}

\subsection{Estimates for \texorpdfstring{$P_n^{\text{mid}}$}{P_n^mid}}

We now bound  $P_n^{\text{mid}}$ as defined as in \eqref{eq:JK-4.4}. 
\begin{proposition}\label{prop:mid_factor_prop}
For $n=sq+r$ we have
    \begin{equation}
    \begin{split}
        \bigg|
        \frac{1}{n}\log |P^{\mathrm{mid}}_n(x,E)|
        -&\int_{\mathbb{R}}\log\lvert E-E'\rvert_{[B_-,B_+^{\log(2+|E|)}]}dN(E')
        \bigg| \\
        &        \le 2000\,n_0\,\frac{s+\lvert r\rvert}{n}(\log(2+|E|)\log B_+-\log B_-).
    \end{split}
    \end{equation}
\end{proposition}
\begin{proof}
Noting that
$$|P_n^{\text{mid}}(x,E)|= \prod_{j=0}^{n-1} |E-E_j(x)|_{[B_-,B_+^{\log(2+|E|)}]}$$
respectively
$$\log|P_n^{\text{mid}}(x,E)|= \sum_{j=0}^{n-1} \log |E-E_j(x)|_{[B_-,B_+^{\log(2+|E|)}]}$$
we can bound the difference using Corollary~\ref{cor:JK-3.7} for the BV-function
\[
    g(E')=\begin{cases}
        0 & |E-E'|<B_-\\
        \log |E-E'| & |E-E'|\in \left[B_-,B_+^{\log(2+|E|)}\right]\\
        0 & |E-E'|>B_+^{\log(2+|E|)}
    \end{cases},
\]
which satisfies $\|g\|_{BV}= 4\log(\frac{B_+^{\log(2+|E|)}}{B_-})$.
\end{proof}
    Suppose that $q(n)$ is the largest continued fraction denominator with the property
    \begin{equation}\label{eq:q(n)_def}
        n=sq(n)+r(n),\qquad 0\leq |r(n)|\leq\sqrt{n}, \qquad s\geq 1,
    \end{equation}
    and also assume $q(n)\geq n_0$. 
    Consequently, we have
    \[
\frac{s+|r|}{n}\le \frac{1}{q(n)}+\frac{2|r(n)|}{n}\le \frac{1}{q(n)}+\frac{2}{\sqrt{n}}.
\]
Under the assumptions of Proposition \ref{prop:mid_factor_prop}, we have
\begin{equation}\label{eq:middle_factor_bound}
\begin{aligned}
\left|\frac{1}{n}\log |P_n^{\mathrm{mid}}(x,E)|-L(E)\right|
&\le L_{\mathrm{corr}}(E,B_-,B_+) \\
			&\quad +2000n_0\left(\frac{1}{q(n)}+\frac{2|r(n)|}{n}\right)
(\log(2+|E|)\log B_+-\log B_-).
\end{aligned}
\end{equation}

where
$$
    L_{\text{corr}}\coloneqq \big\lvert L(E;[B_-,B_+^{\log(2+\lvert E\rvert)}]) - L(E)\big\rvert.
$$

Recall that by Lemma \ref{lemma:lyapunov_uniformity},
\[
L_{\mathrm{corr}}(E,B_-,B_+)
=o(1)(1+L(E))
\]
uniformly in $E$ as $B_-\to0$ and $B_+\to\infty$.
In addition, since we always have $|P_n^{\text{small}}(x,E)|\leq 1$, \eqref{eq:middle_factor_bound} leads to 
\begin{multline}\label{eq:middle_lower_bound}
\bigl|P_n^{\text{small}}(x,E)\,P_n^{\mathrm{mid}}(x,E)\bigr|
\le \exp\Biggl\{
nL(E)+ n\Biggl(
L_{\mathrm{corr}}(E,B_-,B_+)  \\
		+2000n_0\left(\frac{1}{q(n)}+\frac{2|r(n)|}{n}\right)
(\log(2+|E|)\log B_+-\log B_-)
\Biggr)
\Biggr\}.
\end{multline}

\section{Determinant estimates}\label{AppendixC}

In this Section we prove \cite[Lemma 6.2 and Lemma 6.3]{JK2} which are crucial for the determinant estimates needed to establish regularity.

\begin{lemma}[{\cite[Lemma 6.1]{JK2}}]\label{lemma:JK-6.1}
Let $p(x)=x^n+a_{n-1}x^{n-1}+\cdots+a_0$ be a polynomial with real coefficients, and let
$x_0,\dots,x_n\in\mathbb{R}$ satisfy $|x_i-x_j|\geq d$ for $i\ne j$.
Then, at least one of the above points satisfies
\begin{equation}\tag{6.3}
|p(x_j)| \geq \left(\frac{d}{2}\right)^n n!.
\end{equation}
\end{lemma}
\begin{lemma}[{\cite[Lemma 6.2]{JK2}}]\label{lemma:JK-6.2}
Let $I\subset\mathbb{R}$ be an interval with $|I|<(10\gamma)^{-1}$, and let
$x_0,\dots,x_l\in I$ satisfy $|x_i-x_j|\geq d>0$ for $i\ne j$.
Assume that $I$ contains at most $t\leq l$ intersection points \eqref{eq:intersection_pts} for $P_n$,
and all other intersection points $z$ of $P_n$ satisfy
\begin{equation}\tag{6.4}
\mathrm{dist}(z,I)\geq \gamma^{-1}B_- .
\end{equation}
Then, for at least one $j\in[0,l]$, we have
\begin{equation}\tag{6.5}
\bigl|P_n^{\text{small}}(x_j)\bigr|\geq \gamma^{t}\left(\frac{d}{2}\right)^{t}t!.
\end{equation}
\end{lemma}

\begin{proof}
    Recall that we have
    \[
        |P_n^{\text{small}}(x)|\geq\prod\limits_{k:\gamma\|x-z_k(E)\|<B_-}\gamma\|x-z_k(E)\|.
    \]
    Combining this with the fact that $I$ contains at most $t$ intersection points, and that all other intersection points satisfy $\dist(z_k(E),I)\geq \frac{B_-}{\gamma}$ (and so for $x\in I$, $\gamma\|x-z_k(E)\|\geq B_-$). Due to the fact that $|I|\leq\frac{1}{10\gamma}$, we also have that for all intersection points inside $I$ and all $x\in I$, $\gamma\|x-z_k(E)\|<1$. Hence, we obtain
    \[
        |P_n^{\text{small}}(x)|\geq\gamma^t\prod\limits_{k:z_k(E)\in I}\|x-z_k(E)\|.
    \]
    The result now follows from Lemma \ref{lemma:JK-6.1}.
\end{proof}
Let $q(n),r(n),s(n)$ be as in \eqref{eq:q(n)_def}. Define $B_+(n)=q(n)$, $B_+'(n)=q(n)+4$. Denote $P_n^{\text{large}}$ and $P_n^{\text{large}'}$ to be as in \eqref{eq:pn_large_def} with $B_+(n),B_+'(n)$ respectively.
\begin{lemma}[{\cite[Lemma 6.3]{JK2}}]\label{lemma:JK-6.3}
Let $0\leq n'\leq n$, and suppose that
\begin{equation}\label{eq:JK_6.8}
[h';\,h'+n']\subset [h;\,h+n]
\end{equation}
are two integer intervals. Then
\[
\left|
\frac{P_{n'}^{\text{large}'}(x+h'\alpha)}{P_{n}^{\text{large}}(x+h\alpha)}
\right|
\leq
\exp\!\left(
n\log(2+|E|)o(1)\right),\qquad n\to\infty,
\]
uniformly in $E\in\mathbb{R}$.
\end{lemma}
\begin{proof}
    By \eqref{eq:JK_4.10}, we have
\[
P_{n'}^{\text{large}'}(x+h'\alpha,E)\leq\prod\limits_{k:|f(x+(h'+k)\alpha)-E|>B_+'(n)^{\log(2+|E|)}+2}\left(|f(x+(h'+k)\alpha)-E|+2\right)
\]
and
\[
P_n^{\text{large}}(x+h\alpha,E)\geq\prod\limits_{k:|f(x+(h'+k)\alpha)-E|>B_+(n)^{\log(2+|E|)}-2}\left(|f(x+(h'+k)\alpha)-E|-2\right).
\]
Denoting
\[
S:=\Bigl\{0\le k\le n'-1:\ |f(x+(h'+k)\alpha)-E|>B_+(n)^{\log(2+|E|)} +2\Bigr\},
\]
we obtain
\[
\left|\frac{P_{n'}^{\text{large}'}(x+h'\alpha)}{P_{n}^{\text{large}}(x+h\alpha)}
\right|\leq\frac{\prod\limits_{k\in S}\left(|f(x+(h'+k)\alpha)-E|-2\right)}{\prod\limits_{k\in S}\left(|f(x+(h'+k)\alpha)-E|+2\right)}\leq\left(\frac{q(n)+4}{q(n)}\right)^{\log(2+|E|)\left[4\left(s+|r|\right)+n\frac{n_0 C(f)}{\log B_+(n)}\right]},
\]
where the last inequality follows from Proposition \ref{prop:num_of_large_entries}.
\end{proof}

\section{Derivation of lemmas \texorpdfstring{\ref{lemma:diophantine_transition}, \ref{lemma:diophantine_transition_case_2},
\ref{lemma:Liouville_transition_nonresonant} and \ref{lemma:Liouville_transition_resonant}}{2.7, 2.8, 2.9 and 2.10}}

Given the results established in Appendices~A, B, and~C, the proofs of the regularity lemmas
Lemmas~\ref{lemma:diophantine_transition},\ref{lemma:diophantine_transition_case_2}, \ref{lemma:Liouville_transition_nonresonant}, and
\ref{lemma:Liouville_transition_resonant}
all follow the same mechanism as in \cite[Section 7]{JK2}. The difference between the Diophantine/Liouville (and resonant/nonresonant)
regimes is in the arithmetic step
which guarantees that the hypotheses of the determinant estimates from Appendix \ref{AppendixC} are satisfied.

Fix an interval $I=[h_j,h_j+n-1]\cap\mathbb{Z}$ and a site $m\in I$ with $\dist(m,\partial I)\ge \sigma n$.
	Whenever $P_n(x+h_j\alpha,E)\neq0$, Cramer's rule gives
\begin{equation}\label{eq:cramer_split}
\bigl|(H_{[h_j,h_j+n-1]}(x)-E)^{-1}(m,h_j+n-1)\bigr|
=
\left|\frac{P_{m-h_j}(x+h_j\alpha,E)}{P_n(x+h_j\alpha,E)}\right|.
\end{equation}
We then split determinants into parts using cutoffs $B_-(n)$ and $B_+(n)$,
\[
P_k(\cdot,E)=P_k^{\text{small}}(\cdot,E)\,P_k^{\mathrm{mid}}(\cdot,E)\,P_k^{\text{large}}(\cdot,E),
\]
	and split the determinant into three factors (cf.\ (6.11) in \cite{JK2}):
\begin{align}\label{eq:611_split}
\bigl|(H_{[h_j,h_j+n-1]}(x)-E)^{-1}(m,h_j+n-1)\bigr|
&=
\left|
\frac{
P_{m-h_j}^{\mathrm{small}}(x+h_j\alpha,E)\,
P_{m-h_j}^{\mathrm{mid'}}(x+h_j\alpha,E)
}{
P_n^{\mathrm{mid}}(x+h_j\alpha,E)
}
\right|
\nonumber\\
&\qquad\cdot
\left|
\frac{1}{P_n^{\mathrm{small}}(x+h_j\alpha,E)}
\right|
\cdot
\left|
\frac{
P_{m-h_j}^{\mathrm{large'}}(x+h_j\alpha,E)
}{
P_n^{\mathrm{large}}(x+h_j\alpha,E)
}
\right|.
\end{align}
	Here, $P_{m-h_j}^{\mathrm{mid'}}$ is defined as $P_{m-h_j}^{\mathrm{mid}}$ but with upper cutoff $B_+'(n)^{\log(2+|E|)}$ instead of $B_+(n)^{\log(2+|E|)}$. The cutoffs themselves are chosen as $B_-(n)=q(n)^{-2C_d+1}$ with some large but fixed constant $C_d$ and as above $B_+(n)=q(n)$, which allows one to control numerator and denominator using the large deviation bounds up to uniform errors.

	The rest of the argument including the entire arithmetic constructions and interval-selection arguments are now as in \cite[Section~7]{JK2} using the estimates from this appendix:
\begin{itemize}
\item The first factor is controlled using the large deviation bounds from Appendix \ref{AppendixB} (upper bounds for the numerator
and a lower bound for the middle part $P_n^{\mathrm{mid}}$ in the denominator).
\item The second factor is controlled by Lemma~\ref{lemma:JK-6.2} (Appendix \ref{AppendixC}), which provides a quantitative lower bound on $|P_n^{\mathrm{small}}|$ once its hypotheses are met. 
\item The \emph{third factor} is controlled by Lemma~\ref{lemma:JK-6.3} (Appendix \ref{AppendixC}), which bounds the ratio of the large parts.
\item The main difficulty is to control the small eigenvalues. This is done using arithmetic input which is specific to each of the different cases of Lemma~\ref{lemma:diophantine_transition},\ref{lemma:diophantine_transition_case_2}, \ref{lemma:Liouville_transition_nonresonant}, or
\ref{lemma:Liouville_transition_resonant}. In the Diophantine and nonresonant cases, the resulting determinant
estimates give exponential decay of the relevant Green functions at rate $L(E)-o(1)$. In the resonant case an additional argument is required. Following
\cite[Sections~7.7-7.8]{JK2}, one applies the Poisson formula and uses
nonresonant regularity to estimate the boundary values in terms of
maxima of an eigenfunction $|\psi|$ on neighboring resonant regions. This does not give exponential decay of the Green function, but gives the claimed $\psi$-regularity with decay rate $L(E)-\beta(\alpha)-o(1)$.
\end{itemize}

\noindent\small{{\bf Acknowledgements and AI statement}
This research was supported by NSF DMS-2052899, DMS-2155211, and Simons
896624}.

No AI tools were used in developing the results or preparing the manuscript until the final pre-submission stage. At that stage,  LLM was used to stress-test the arguments, leading to a number of minor corrections and improvements in wording, and to prepare Figure~\ref{fig:sturmian_comparison}. The mathematical results and proofs were complete by May 2026, and  this work was presented in the third named author's talk at the 39th Western States Mathematical Physics Meeting. Posting was delayed by the first named author’s overly optimistic assessment of her available time.

\end{document}

%% file: figure_smoothing.tex
\scalebox{0.75}{
\begin{tikzpicture}[
  x=0.5pt,y=-0.5pt,
  every node/.style={inner sep=0pt,outer sep=0pt},
  axis/.style={draw=axisgray,line width=0.3pt,
    -{Latex[length=6pt,width=4.5pt]}},
  tick/.style={draw=axisgray,line width=0.3pt},
  guide/.style={draw=guidegray,line width=0.25pt},
  endpointguide/.style={draw=axisgray,line width=0.3pt,
    dash pattern=on 3.5pt off 3.5pt},
  graph/.style={line width=1.5pt,line cap=round,line join=round},
  heading/.style={font=\fontsize{13.8}{17}\selectfont\bfseries},
  conclusion/.style={font=\fontsize{13.8}{17}\selectfont\bfseries},
  labeltext/.style={font=\fontsize{13}{15.6}\selectfont},
  brace/.style={draw=axisgray,line width=0.5pt,line cap=round}
]
  \definecolor{sturmianblue}{RGB}{37,79,120}
  \definecolor{smoothteal}{RGB}{8,127,128}
  \definecolor{axisgray}{RGB}{91,98,105}
  \definecolor{guidegray}{RGB}{225,227,228}
  \definecolor{intervalyellow}{RGB}{245,232,199}



  \node[heading,text=sturmianblue] at (278,91)
    {(a) Sturmian sampling};
  \node[heading,text=smoothteal] at (861,91)
    {(b) Smoothed sampling};
  \node[labeltext] at (278,135)
    {$f_0=\lambda\mathbf{1}_{[1-\alpha,1)}$};
  \node[labeltext] at (861,135)
    {$f_\varepsilon\in C^1([0,1))$};

  \foreach \offset/\ylabel in {0/{f_0(x)},583/{f_\varepsilon(x)}} {
    \begin{scope}[xshift=\offset*0.5pt]
      \fill[intervalyellow] (210,200) rectangle (247,397);
      \draw[guide] (69,205) -- (486,205);
      \draw[guide] (69,395) -- (486,395);
      \draw[endpointguide] (486,205) -- (486,395);
      \draw[axis] (50,417) -- (518,417);
      \draw[axis] (50,417) -- (50,177);
      \draw[tick] (46,205) -- (55,205);
      \draw[tick] (46,395) -- (55,395);
      \foreach \tickx in {69,229,486}
        \draw[tick] (\tickx,417) -- (\tickx,423);
      \node[labeltext] at (50,155) {$\ylabel$};
      \node[labeltext] at (27,207) {$\lambda$};
      \node[labeltext] at (27,395) {$0$};
      \node[labeltext] at (69,443) {$0$};
      \node[labeltext] at (229,443) {$1-\alpha$};
      \node[labeltext] at (486,443) {$1$};
      \node[labeltext] at (535,417) {$x$};

      \draw[brace] (210,195)
        .. controls (214,190) and (217,192) .. (221,191)
        .. controls (226,192) and (226,187) .. (229,187)
        .. controls (232,187) and (232,192) .. (237,191)
        .. controls (241,192) and (244,190) .. (247,195);
      \node[labeltext] at (229,169) {$\varepsilon$};
    \end{scope}
  }

  \draw[sturmianblue,line width=0.55pt,
    dash pattern=on 3.5pt off 3.5pt] (229,205) -- (229,395);
  \draw[graph,sturmianblue] (69,395) -- (229,395);
  \draw[graph,sturmianblue] (229,205) -- (486,205);
  \foreach \point in {(69,395),(229,205),(486,395)}
    \fill[sturmianblue] \point circle[radius=2pt];
  \foreach \point in {(229,395),(486,205)}
    \draw[sturmianblue,fill=white,line width=1.1pt]
      \point circle[radius=2.3pt];

  \begin{scope}[xshift=291.5pt]
    \draw[graph,smoothteal]
      (69,395) -- (210,395)
      .. controls (219,395) and (238,205) .. (247,205)
      -- (486,205);
    \foreach \point in {(69,395),(486,395)}
      \fill[smoothteal] \point circle[radius=2pt];
    \draw[smoothteal,fill=white,line width=1.1pt]
      (486,205) circle[radius=2.3pt];
  \end{scope}

  \node[conclusion,text=sturmianblue] at (278,487)
    {Purely singular continuous spectrum};
  \node[conclusion,text=smoothteal] at (861,487)
    {Anderson localization};
  \node[labeltext] at (278,523) {every phase};
  \node[labeltext] at (861,523) {almost every phase};
\end{tikzpicture}
}